\documentclass{siamltex}

\pdfoutput=1

\usepackage[T1]{fontenc}
\usepackage[latin1]{inputenc}
\usepackage{latexsym}
\usepackage[pdftex]{graphicx}
\usepackage{stmaryrd}
\usepackage{amsmath,amssymb}
\usepackage[activate={true,nocompatibility}]{microtype}
\usepackage{remreset}
\usepackage[pdftex,
            colorlinks,
            hyperfootnotes=false,
            pdffitwindow=true,
            plainpages=false,
            pdfpagelabels=true,
            pdfpagemode=UseOutlines,
            pdfpagelayout=SinglePage,
            pdftitle={Numerical Evaluation of Fredholm Determinants},
            pdfauthor={Folkmar Bornemann, Technische Universität München},
            pagebackref,
            hyperindex,
            filecolor=purple,
            urlcolor=purple,
            citecolor=blue,
            linkcolor=blue]{hyperref}
\usepackage{myharvard}
\usepackage{bm}
\usepackage[sc,osf]{mathpazo}

\makeatletter
\@removefromreset{figure}{section}
\makeatother

\input{macros.sty}

\title{On the Numerical Evaluation of Fredholm Determinants}

\author{Folkmar Bornemann\thanks{Zentrum Mathematik, Technische Universität München,
        Boltzmannstr. 3, 85747 Garching, Germany ({\tt bornemann@ma.tum.de}). Manuscript
        as of \today.}}

\begin{document}

\maketitle

\begin{abstract} Some significant quantities in mathematics and physics are most naturally expressed as the Fredholm determinant of an integral operator, most notably
many of the distribution functions in random matrix theory. Though their numerical values are of interest, there is no systematic numerical treatment of Fredholm determinants
to be found in the literature. Instead, the few numerical evaluations that are available rely on eigenfunction expansions of the operator, if expressible in terms of special functions,
or on alternative, numerically more straightforwardly accessible analytic expressions, e.g., in terms of Painlevé transcendents, that have masterfully been derived in some cases.
In this paper we close the gap in the literature by studying projection methods and, above all, a simple, easily implementable, general method for the numerical evaluation
of Fredholm determinants that is derived from the classical
Nyström method for the solution of Fredholm equations of the second kind. Using Gauss--Legendre or Clenshaw--Curtis as the underlying quadrature rule,
we prove that the approximation error essentially behaves like the quadrature error for the sections of the kernel. In particular, we get exponential convergence for
analytic kernels, which are typical in random matrix theory.
The application of the method to the distribution functions of the Gaussian unitary ensemble (GUE), in the bulk and the edge scaling limit, is discussed in detail.
After extending the method to systems of integral operators, we evaluate the two-point correlation functions of the more recently studied Airy and $\text{Airy}_1$ processes.
\end{abstract}

\section{Introduction}\label{sect:intro}

\possessivecite{34.0422.02} landmark paper\footnote{\citeasnoun[p.~VII]{MR2154153} writes:
``This deep paper is extremely readable and I recommend it to those wishing a pleasurable afternoon.''
An English translation of the paper can be found in \citeasnoun[pp.~449--465]{MR0469612}.} on linear integral
equations is generally considered to be the forefather of those mathematical concepts that
finally led to modern functional analysis and operator theory---see the historical accounts in
\citeasnoun[pp.~1058--1075]{MR0472307} and \citeasnoun[Chap.~V]{MR605488}. Fredholm was interested
in the solvability of
what is now called a Fredholm equation of the second kind,
\begin{equation}\label{eq:fred1}
u(x) + z \int_a^b K(x,y) u(y)\,dy = f(x) \qquad (x \in(a,b)),
\end{equation}
and explicit formulas for the solution thereof,
for a right hand side $f$ and a kernel $K$, both assumed to be continuous functions. He introduced his now famous
determinant
\begin{equation}\label{eq:det1}
d(z) = \sum_{k=0}^\infty \frac{z^n}{n!} \int_a^b \cdots \int_a^b \det\left(K(t_p,t_q)\right)_{p,q=1}^n \,dt_1\cdots\,dt_n,
\end{equation}
which is an entire function of $z\in\C$, and succeeded in showing that the integral equation is uniquely solvable if and only if $d(z)\neq 0$.

Realizing the tremendous potential of Fredholm's theory, Hilbert started working on integral equations in a flurry and, in a series
of six papers from 1904 to 1910,\footnote{Later reproduced as one of the first books on linear integral equations \cite{43.0423.01}.} transformed the determinantal framing to
the beginnings of what later, in the hands of Schmidt, Carleman, Riesz, and others, would become the theory of compact operators in Hilbert spaces. Consequently, over the years Fredholm determinants
have faded from the core of general accounts on integral equations to the historical remarks section---if they are mentioned at all.\footnote{For example, \citeasnoun{43.0423.01}, \citeasnoun{53.0453.01}, and \citeasnoun{53.0180.04}
start with the Fredholm determinant right from the beginning; yet already \citeasnoun[pp.142--147]{MR0065391}, the translation of the German edition from  1931, give it just a short mention
(``since we shall not make any use of the Fredholm formulas later on''); while \citeasnoun[Chap.~V]{MR0104991} and \citeasnoun[Chap.~7]{MR0390680}, acknowledging
the fact that ``classical'' Fredholm theory yields a number of results that functional analytic techniques do not, postpone Fredholm's theory to a later chapter;
whereas \citeasnoun{MR467215}, \citeasnoun{MR1111247}, \citeasnoun{MR1193030},
\citeasnoun{MR1350296}, and \citeasnoun{MR1723850}
ignore the Fredholm determinant altogether. Among the newer books on linear integral equations, the monumental four volume work of 
\citeasnoun{Fenyo} is one of the few we know of
 that give Fredholm determinants a balanced treatment.}

So, given this state of affairs, why then study the numerical evaluation of the Fredholm determinant $d(z)$?
The reason is, simply enough, that the Fredholm determinant and the more general notions, generalizing (\ref{eq:fred1}) and (\ref{eq:det1}) to
\[
u + z A u = f,\qquad  d(z) = \det(I+z A),
\]
for certain classes of compact operators $A$ on Hilbert spaces, have always remained important tools in operator theory and mathematical physics \cite{MR1744872,MR2154153}.
In turn, they have found many significant applications: e.g.,
in atomic collision theory \cite{MR0044404,Moi77}, inverse scattering \cite{MR0406201}, in Floquet theory of periodic differential equations \cite{0287.34016}, in the infinite-dimensional method of stationary phase
and Feynman path integrals \cite{MR0474436,MR1284645},
as the two-point correlation function of the
two-dimensional Ising model \cite{Wilk78},
in renormalization in quantum field theory \cite{MR2154153},
as distribution functions in random matrix theory \cite{MR2129906,MR1677884,MR1659828}
and combinatorial growth processes \cite{Johan00,MR1933446,Sasa05,MR2363389}. As \citeasnoun[p.~260]{MR1892228} puts it most aptly  upon including
Fredholm's theory as a chapter of its own in his recent textbook on functional analysis: ``Since this determinant appears in some modern theories, it is time to resurrect it.''

In view of this renewed interest in operator determinants, what numerical methods are available for their evaluation? Interestingly, this question has apparently never---at least to our knowledge---been systematically addressed
in the numerical analysis literature.\footnote{Though we can only speculate
about the reasons, there is something like a disapproving attitude towards determinants in general that seems to be quite common among people working in ``continuous'' applied mathematics. Here are a few scattered examples: \citeasnoun{MR17382} writes at the beginning of the chapter on determinants (p.~460): ``Today matrix and linear algebra are in the main stream of applied mathematics, while the role
of determinants has been relegated to a minor backwater position.'' \citeasnoun{Axler95} has a paper with the provocative title ``Down with Determinants!'' and a correspondingly worked out textbook on linear algebra \citeyear{Axler97}.
The quintessential book of \citeasnoun{MR1417720} on ``Matrix Computations'' does not explicitly address the computation of determinants at all, it is only implicitly stated as part of Theorem~3.2.1.
\citeasnoun{MR1927606} writes at the beginning of Section 14.6: ``Like the matrix inverse, the determinant is a quantity that rarely needs to be computed.'' He then continues with the argument, well known to every numerical
analyst, that the determinant cannot be used as a measure of ill conditioning since it scales as $\det(\alpha A) = \alpha^m \det(A)$  for a $m\times m$-matrix $A$, $\alpha \in \R$.
Certainly there is much truth in all of their theses, and Cramer's rule and the characteristic polynomial, which were the most common reasons for a call
to the numerical evaluation of determinants \cite[p.~176]{MR1653546}, have most righteously been banned from the toolbox of a numerical analyst for reasons of efficiency.
However, with respect to the infinite dimensional case, the elimination of determinants from the thinking of numerical analysts as a subject of computations might have been all too successful. For instance, the
scaling argument does not apply in the infinite dimensional case: operator determinants $\det(I+A)$ are defined for compact perturbations of the identity,
which perfectly determines the scaling since, for $\alpha \neq 1$, $\alpha (I+A)$ cannot
be written in the form $I+\tilde A$ with another compact operator $\tilde A$. (This is because the identity operator is \emph{not} compact then.)}
Even experts in the applications of Fredholm determinants commonly seem to have been thinking \cite{Spohn1} that an evaluation is only possible
if either the eigenvalues of the integral operator are, more or less, explicitly known or if an alternative analytic expression has been found that is numerically more accessible---in each specific case anew,
lacking a general procedure.

\medskip

\paragraph{The Nyström-type method advocated in this paper}
In contrast, we study a simple general numerical method for Fredholm determinants which is exceptionally efficient for smooth
kernels, yielding small \emph{absolute} errors (i.e., errors that are small with respect to the scale $\det(I)=1$ inherently given by the operator determinant).
To this end we follow the line of thought of \possessivecite{Ny30} classical quadrature method for the numerical solution of the Fredholm equation (\ref{eq:fred1}).
Namely, given a quadrature rule
\[
Q(f) = \sum_{j=1}^m w_j\, f(x_j) \approx \int_a^b f(x)\,dx,
\]
\citename{Ny30} discretized (\ref{eq:fred1}) as the linear system
\begin{equation}\label{eq:ny1}
u_i + z \sum_{j=1}^m w_j K(x_i,x_j) u_j = f(x_i)\qquad (i=1,\ldots,m),
\end{equation}
which has to be solved for $u_i \approx u(x_i)$ ($i=1,\ldots,m)$. Nyström's method is extremely simple and, yet, extremely effective for \emph{smooth} kernels. So much so that
\citeasnoun[p.~245]{MR837187}, in a chapter comparing
different numerical methods for Fredholm equations of the second kind, write:
\medskip\begin{quote}
Despite the theoretical and practical care lavished on the more complicated algorithms, the clear winner of this contest has been the Nyström routine with the $m$-point Gauss--Legendre rule. This routine is extremely simple;
it includes a call to a routine which provides the points and weights for the quadrature rule, about twelve lines of code to set up the Nyström equations and a call to the routine which solves these equations. Such results
are enough to make a numerical analyst weep.
\end{quote}\medskip

\noindent
By keeping this conceptual and algorithmic simplicity,
the method studied in this paper approximates the
Fredholm determinant $d(z)$ simply by the determinant of the $m\times m$-matrix that is applied to the vector $(u_i)$ in the Nyström equation~(\ref{eq:ny1}):
\begin{equation}\label{eq:ny}
d_Q(z) = \det \left(\delta_{ij} + z\,w_i K(x_i,x_j)\right)_{i,j=1}^m.
\end{equation}
If the weights $w_j$ of the quadrature rule are positive (which is always the better choice), we will use the \emph{equivalent} symmetric variant
\begin{equation}\label{eq:ny2}
d_Q(z) = \det\left(\delta_{ij} + z\, w_i^{1/2} K(x_i,x_j) w_j^{1/2}\right)_{i,j=1}^m.
\end{equation}
Using Gauss--Legendre or Curtis--Clenshaw quadrature rules, the computational cost\footnote{The computational cost of $O(m^3)$ for the matrix determinant using either
Gaussian elimination with partial pivoting \cite[p.~176]{MR1653546} or Hyman's
method \cite[Sect.~14.6]{MR1927606} clearly dominates the cost of $O(m\log m)$ for the weights and points of Clenshaw--Curtis quadrature using the FFT,
as well as the cost of $O(m^2)$ for Gauss--Legendre quadrature using the Golub--Welsh algorithm; for implementation details of these quadrature rules see \citeasnoun{MR2214855}, \citeasnoun{Tref08}, and
the appendix of this paper.
The latter paper carefully compares Clenshaw--Curtis with Gauss--Legendre and concludes (p.~84): ``Gauss quadrature is a beautiful and powerful idea. Yet the Clenshaw--Curtis formula has essentially the same performance
for most integrands.''}
of the method is of order $O(m^3)$. The implementation in Matlab, or Mathematica,
is straightforward and takes just a few lines of code.\footnote{The command {\tt [w,x] = QuadratureRule(a,b,m)} is supposed to supply the weights and
points of a $m$-point quadrature rule on the interval $[a,b]$ as a $1\times m$ vector {\tt w} and a $m\times 1$-vector {\tt x}, respectively. For Gauss--Legendre and Clenshaw--Curtis,
such a Matlab code
can be found in the appendix.} In Matlab:

\medskip\begin{quote}
\begin{verbatim}
function d = DetNyström(K,z,a,b,m)
[w,x] = QuadratureRule(a,b,m);
w = sqrt(w);
[xj,xi] = meshgrid(x,x);
d = det(eye(m)+z*(w'*w).*K(xi,xj));
\end{verbatim}
\end{quote}\medskip

\noindent In Mathematica:
\smallskip

\begin{center}\label{prog:math}
\hspace*{1mm}\includegraphics[width=0.835\textwidth]{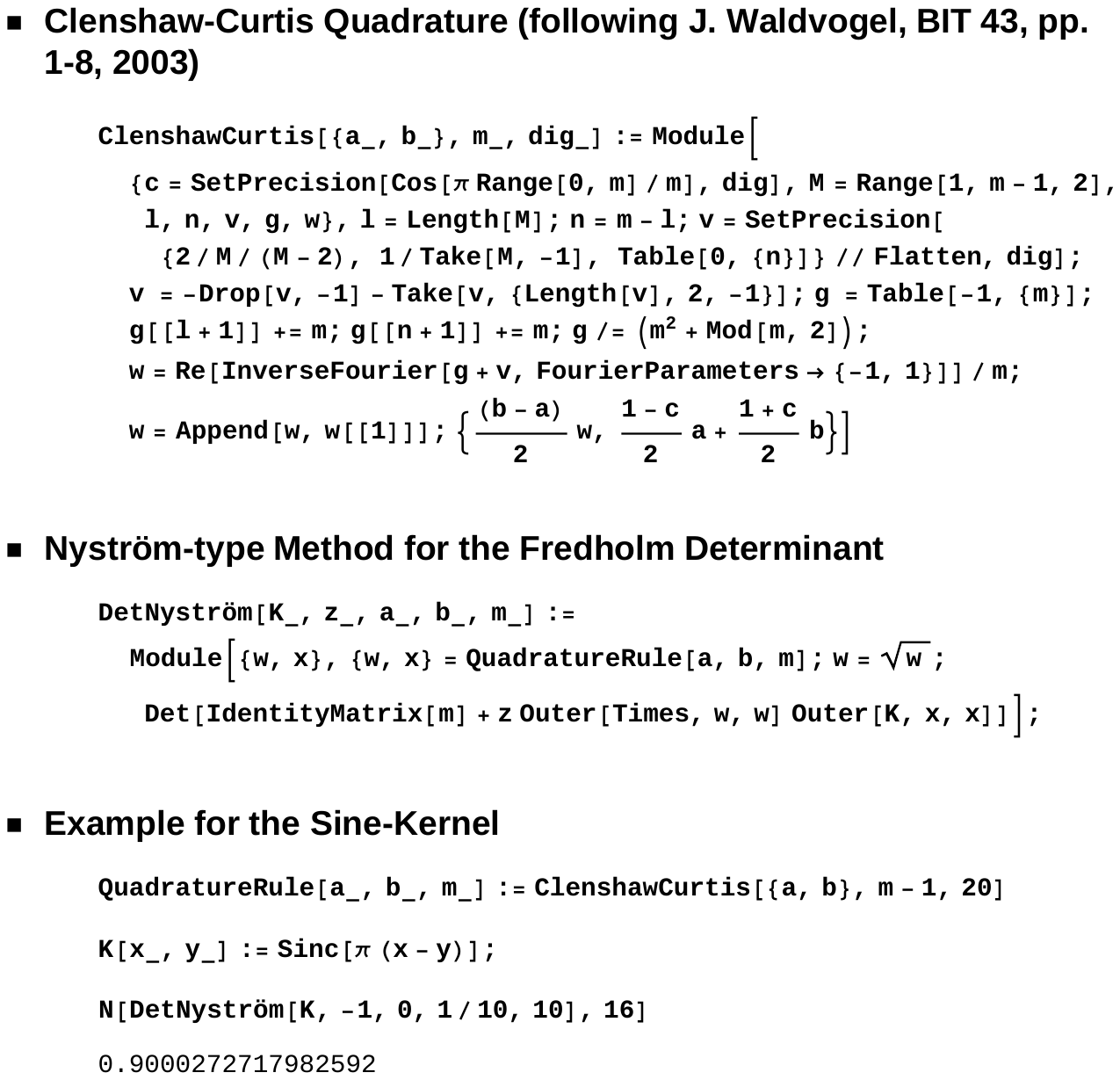}
\end{center}

Strictly speaking we are not the first to suggest this simple method, though. In fact, it was \citename{35.0378.02} himself \citeyear[pp.~52--60]{35.0378.02} in his very first paper on
integral equations, who motivated\footnote{Fredholm himself does not give the slightest hint of a motivation in his early papers \citeyear{Fred00,34.0422.02}.
He apparently conceived his determinant in {\em analogy} to similar expressions that his fellow
countryman \citeasnoun{vonKoch} had obtained for infinite matrices a few years earlier; see \citeasnoun[p.~95]{Fred09}, \citeasnoun[p.~99]{MR605488},
or \citeasnoun[p.~409]{MR2300779}.} the expression (\ref{eq:det1}) of the
Fredholm determinant by essentially using this method with the \emph{rectangular rule} for quadrature, proving locally uniform convergence;
see also \citeasnoun[pp.~4--10]{43.0423.01} and, for the motivational argument given just heuristically,
without a proof of convergence, \citeasnoun[Sect.~11.2]{53.0180.04}, \citeasnoun[pp.~66--68]{MR0094665} (who speaks of a ``poetic license'' to be applied ``without too many scruples''),
\citeasnoun[pp.~65--68]{MR0104991}, \citeasnoun[pp.~243--239]{MR0390680}, and \citeasnoun[Vol.~II, pp.~82--84]{Fenyo}---to
name just a few but influential cases.
Quite astonishingly, despite of all its presence as a motivational tool in the expositions of the classical theory, we have found just one example of the use of this method (with Gauss--Legendre quadrature) in an actual
numerical calculation: a paper by
the physicists \citeasnoun{MR0266459} on
low-energy elastic scattering of electrons from hydrogen atoms.
However, the error estimates (Theorem~\ref{thm:nyerr}) that we will give in this paper seem to be new at least;
we will prove that the approximation error essentially behaves like the quadrature error for the sections $x \mapsto K(x,y)$ and $y \mapsto K(x,y)$ of the kernel. In particular, we will
get exponential
convergence rates for analytic kernels.

\medskip
\paragraph{Examples}\label{ex:intro}
Perhaps the generality and efficiency offered by our direct numerical approach to Fredholm determinants, as compared to analytic methods if they are available at all, is
best appreciated by an example. The probability $E_2(0;s)$ that an interval of length $s$ does not contain, in the bulk scaling limit of level spacing $1$, an eigenvalue
of the Gaussian unitary ensemble (GUE) is given by the Fredholm determinant of the sine kernel,
\[
E_2(0;s) = \det\left(I - A_s\projected{L^2(0,s)}\right),\qquad A_s u(x) = \int_0^s \frac{\sin(\pi(x-y))}{\pi(x-y)} u(y)\,dy\,;
\]
see \citeasnoun{Gau61} and \citeasnoun[Sect.~6.3]{MR2129906}. \citename{Gau61} has further shown that the eigenfunctions of this selfadjoint integral operator
are exactly given by a particular family of special functions, namely the radial prolate spheroidal wave functions with certain parameters. Using tables \cite{MR0074130} of these functions  he was finally able to evaluate $E_2(0;s)$
numerically. On the other hand, in an admirably intricate analytic {tour de force} \citeasnoun{MR573370} expressed the Fredholm determinant of the sine kernel as
\begin{equation}\label{eq:jimbo}
E_s(0;s) = \exp\left(\int_0^{\pi s} \frac{\sigma(x)}{x}\,dx\right)
\end{equation}
in terms of the sigma, or Hirota, representation of the fifth Painlevé equation, namely
\[
(x \sigma'')^2 + 4(x \sigma' - \sigma)(x \sigma' - \sigma + \sigma'^2) = 0,\qquad \sigma(x) \sim -\frac{x}{\pi} - \frac{x^2}{\pi^2} \quad(x \to 0),
\]
see also \citeasnoun[Chap.~21]{MR2129906} and \citeasnoun[Sect.~4.1]{MR1791893}. With respect to the numerical evaluation,
the latter two authors conclude in a footnote, most probably by comparing to Gaudin's method: ``Without the Painlevé representations, the numerical evaluation of the Fredholm determinants is quite involved.'' However,
one does not need to know more than the smooth kernel $\sin(\pi(x-y))/(\pi(x-y))$ ifself to approximate $E_2(0;s)$ with the method of this paper. For instance,
the Gauss--Legendre rule with just $m=5$ quadrature points already gives, in 0.2\,ms computing time,
15~accurate digits of the value
\[
E_2(0,0.1) = 0.90002\,72717\,98259\,\cdots,
\]
that is, by calculating the determinant of a $5\times 5$-matrix easily built from the kernel.

Even though it is satisfying to have an alternative and simpler way of calculating already known quantities, it is far more exciting to be able to calculate quantities that otherwise have defeated numerical evaluations
so far.
For instance, the joint distribution functions of the Airy and the $\text{Airy}_1$ processes are given as determinants of systems of integral operators, see
\citeasnoun{MR1933446}, \citeasnoun{MR2018275}, \citeasnoun{Sasa05} and
\citeasnoun{MR2363389}. Even though a nonlinear partial differential equation of third order in three variables has been found by \citeasnoun[Eq.~(4.12)]{MR2150191} for the logarithm of the joint distribution
function of the Airy process at two different
times, this masterful analytic result is probably of next to no numerical use. And in any case, no such analytic results are yet known for the $\text{Airy}_1$ process. However, the Nyström-type method
studied in this paper can easily
be extended to systems of integral operators. In this way, we have succeeded in evaluating the two-point correlation functions of both stochastic processes,
see Section~\ref{sect:matrixkernels}.

\medskip

\paragraph{Outline of this paper} For the proper functional analytic setting, in Section~\ref{sect:trace} we review some basic facts
about trace class and Hilbert--Schmidt operators. In Section~\ref{sect:det} we review the concept of the determinant $\det(I+A)$ for trace class operators $A$
and its relation to the Fredholm determinant. In Section~\ref{sect:cond} we study perturbation bounds implying that numerical calculations of determinants
can only be expected to be accurate with respect to {\em absolute} errors in general. In Section~\ref{sect:proj} we use the functional analytic formulation of the
problem to obtain convergence estimates for projection methods of Galerkin and Ritz--Galerkin type. The convergence rate is shown to depend on a proper balance between
the decay of the singular values of the operator and the growth of bounds on the derivatives of the corresponding singular functions. This is in sharp contrast with Section~\ref{sect:quad},
where we study the convergence of the Nyström-type method (\ref{eq:ny2}) by directly addressing the original definition of the Fredholm determinant. Here, only the smoothness
properties of the kernel enter the convergence estimates. It turns out that, for kernels of low regularity, the order of convergence of the Nyström-type method can be even higher
than that of a Ritz--Galerkin method. In Section~\ref{sect:random} we give examples for the exponential convergence rates enjoyed by analytic kernels. To this end we discuss
the details of the numerical evaluation of the determinants of the sine and Airy kernels, which express the
probability distributions $E_2(0;s)$ and $F_2(s)$ (the Tracy--Widom distribution) of random matrix theory. Finally, after extending the Nyström-type method to systems of integral
operators we report in Section~\ref{sect:matrixkernels} on the numerical evaluation of the two-point correlation functions of the Airy and $\text{Airy}_1$ processes.

\section{Trace Class and Hilbert--Schmidt Operators}\label{sect:trace}
We begin by recalling some basic material about the spectral theory of nonselfadjoint compact operators, which can be found, e.g., in \citeasnoun{MR1130394}, \citeasnoun{MR1892228} and
\citeasnoun{MR2154153}.
We consider a complex, separable Hilbert space $\mathcal{H}$ with an inner product $\langle\cdot,\cdot\,\rangle$ that is linear in the {\em second} factor and conjugate linear in the first.
The set of bounded linear operators will be denoted by $\mathcal{B}(\mathcal{H})$, the compact operators by $\mathcal{J}_\infty(\mathcal{H})$. The spectrum
of a compact operator $A \in \mathcal{J}_\infty(\mathcal{H})$ has no non-zero limit point; its non-zero points are eigenvalues of finite algebraic multiplicity.
We list these eigenvalues as $(\lambda_n(A))_{n=1}^{N(A)}$, counting multiplicity, where $N(A)$ is either a finite non-negative integer or infinity, and
order them by
\[
|\lambda_1(A)| \geq |\lambda_2(A)| \geq \cdots.
\]
The positive eigenvalues
\[
s_1(A) \geq s_2(A) \geq \cdots > 0
\]
of the associated positive-semidefinite, selfadjoint operator
\[
|A|=(A^* A)^{1/2}
\]
are called the singular values of $A$. Correspondingly, there is the Schmidt or singular-value representation of $A$,
that is, the norm convergent expansion
\begin{equation}\label{eq:schmidt}
A = \sum_{n=1}^{N(|A|)} s_n(A) \langle u_n,\cdot\,\rangle v_n,
\end{equation}
where the $u_n$ and $v_n$ are certain (not necessarily complete) orthonormal sets in $\mathcal{H}$. Note that $s_n(A)=|\lambda_n(A)|$ if $A$ is selfadjoint.
In general we have Weyl's inequality
\begin{equation}\label{eq:weyl}
\sum_{n=1}^N |\lambda_n(A)|^p \;\leq\; \sum_{n=1}^N s_n(A)^p\qquad (N \leq N(A),\;1\leq p< \infty).
\end{equation}
The Schatten--von Neumann classes of compact operators are defined as
\[
\mathcal{J}_p(\mathcal{H}) = \{A \in \mathcal{J}_\infty(\mathcal{H}) \;:\; \sum_{n=1}^{N(|A|)} s_n(A)^p <  \infty \}\qquad (1\leq p < \infty)
\]
with the corresponding norm\footnote{In matrix theory these norms are not commonly used---with the exception of $p=2$: $\|A\|_{\scriptscriptstyle\mathcal{J}_2}$ is then the Schur or Frobenius norm of the matrix $A$.}
\[
\|A\|_{\scriptscriptstyle\mathcal{J}_p} = \left(\sum_{n=1}^{N(|A|)} s_n(A)^p\right)^{1/p}.
\]
The operator norm on $\mathcal{J}_\infty(\mathcal{H})$ perfectly fits into this setting if we realize that
\[
\|A\| = s_1(A) = \max_{n=1,\ldots,N(|A|)} s_n(A) = \|A\|_{\scriptscriptstyle\mathcal{J}_\infty}.
\]
There are the continuous embeddings $\mathcal{J}_p(\mathcal{H}) \subset \mathcal{J}_q(\mathcal{H})$ for $1\leq p \leq q \leq \infty$ with
\[
\|A\|_{\scriptscriptstyle\mathcal{J}_q} \leq \|A\|_{\scriptscriptstyle\mathcal{J}_p}.
\]
The classes $\mathcal{J}_p(\mathcal{H})$ are
two-sided operator ideals in $\mathcal{B}(\mathcal{H})$, that is, for $A \in \mathcal{J}_p(\mathcal{H})$ ($1\leq p\leq\infty$) and $B \in \mathcal{B}(\mathcal{H})$ we have $AB,BA \in \mathcal{J}_p(\mathcal{H})$ with
\begin{equation}\label{eq:ideal}
\|A B\|_{\scriptscriptstyle\mathcal{J}_p} \leq \|A\|_{\scriptscriptstyle\mathcal{J}_p} \|B\|,\qquad \|B A\|_{\scriptscriptstyle\mathcal{J}_p} \leq \|B\|\, \|A\|_{\scriptscriptstyle\mathcal{J}_p}.
\end{equation}
Of special interest to us are the {\em trace class operators} $\mathcal{J}_1(\mathcal{H})$ and the {\em Hilbert--Schmidt operators} $\mathcal{J}_2(\mathcal{H})$. The product of
two Hilbert--Schmidt operators is of trace class:
\[
\|AB\|_{\scriptscriptstyle\mathcal{J}_1} \leq \|A\|_{\scriptscriptstyle\mathcal{J}_2} \|B\|_{\scriptscriptstyle\mathcal{J}_2}\qquad (A,B\in \mathcal{J}_2(\mathcal{H})).
\]
The \emph{trace} of a trace class operator $A$ is defined
by
\[
\tr(A) = \sum_{n=1}^\infty \langle u_n,Au_n\rangle
\]
for any orthonormal basis $(u_n)_n$. A deep theorem of Lidskii's \cite[Chap.~3]{MR2154153} tells us that
\begin{equation}\label{eq:tracedef}
\tr(A) = \sum_{n=1}^{N(A)} \lambda_n(A),
\end{equation}
which implies by Weyl's inequality (\ref{eq:weyl}) that
\begin{equation}\label{eq:sum1lambda}
|\tr(A)| \leq  \sum_{n=1}^{N(A)} |\lambda_n(A)| \leq  \tr(|A|) = \|A\|_{\scriptscriptstyle\mathcal{J}_1}.
\end{equation}
Likewise, for a Hilbert--Schmidt operator $A \in \mathcal{J}_2(\mathcal{H})$ we have
\begin{equation}\label{eq:sum2lambda}
\tr(A^2) = \sum_{n=1}^{N(A)} \lambda_n(A)^2,\qquad |\tr(A^2)| \leq \sum_{n=1}^{N(A)} |\lambda_n(A)|^2 \leq  \|A\|_{\scriptscriptstyle\mathcal{J}_2}^2.
\end{equation}

\medskip
\paragraph{Integral operators with $L^2$-kernel}
In the Hilbert space $\mathcal{H}=L^2(a,b)$ of square-integrable functions on a finite interval $(a,b)$ the Hilbert--Schmidt operators are exactly given by
the integral operators with $L^2$-kernel. That is, there is a one-to-one correspondence \cite[Thm.~2.11]{MR2154153} between $A \in \mathcal{J}_2(\mathcal{H})$ and $K \in L^2((a,b)^2)$
mediated through
\begin{equation}\label{eq:intop}
A u(x) = \int_a^b K(x,y) u(y)\,dy
\end{equation}
with equality of norms $\|A\|_{\scriptscriptstyle\mathcal{J}_2} = \|K\|_{L^2}$: the spaces $\mathcal{J}_2(\mathcal{H})$ and $L^2((a,b)^2)$ are thus isometrically isomorphic.
In particular, by (\ref{eq:sum2lambda}) and a well known basic result on infinite products \cite[p.~232]{MR0183997}, we get for such operators that
\[
\prod_{n=1}^{N(A)} (1+\lambda_n(A)) \text{ converges (absolutely) } \;\Leftrightarrow\; \sum_{n=1}^{N(A)} \lambda_n(A) \text{ converges (absolutely)}.
\]
Since the product is a natural candidate for the definition of $\det(I+A)$ it makes sense requiring $A$ to be of trace class; for then, by (\ref{eq:sum1lambda}),
the absolute convergence of the sum can be guaranteed.

\medskip
\paragraph{Integral operators with a continuous kernel} A continuous kernel $K \in C([a,b]^2)$ is certainly square-integrable. Therefore, the induced integral
operator (\ref{eq:intop}) defines a Hilbert--Schmidt operator $A$ on the Hilbert space $\mathcal{H}=L^2(a,b)$. Moreover, other than for $L^2$ kernels in general, the
integral
\[
\int_a^b K(x,x)\,dx
\]
over the diagonal of $(a,b)^2$ is now well defined and constitutes, in analogy to the matrix case, a ``natural'' candidate for the trace of the integral operator. Indeed,
if an integral operator $A$ with continuous kernel is of trace class, one can prove \cite[Thm.~8.1]{MR1744872}
\begin{equation}\label{eq:tracediag}
\tr(A) = \int_a^b K(x,x)\,dx.
\end{equation}
Unfortunately, however, just the continuity of the kernel $K$ does not guarantee the induced integral operator $A$
to be of trace class.\footnote{A counter-example was discovered by \citeasnoun{Carleman18}, see also \citeasnoun[p.~71]{MR1744872}.}
Yet, there is some encouraging positive experience stated by \citeasnoun[p.~25]{MR2154153}:

\medskip\begin{quote}
However, the counter-examples which prevent nice theorems from holding are generally rather contrived so that I have found the following to be true: If an
integral operator with kernel $K$ occurs in some `natural' way and $\int |K(x,x)|\,dx < \infty$, then the operator can (almost always) be proven to be trace class (although sometimes only
after some considerable effort).
\end{quote}\medskip

\noindent Nevertheless, we will state some simple criteria that often work well:

\smallskip
\begin{enumerate}
\item If the continuous kernel $K$ can be represented in the form
\[
K(x,y) = \int_c^d K_1(x,y) K_2(z,y)\, dz\qquad (x,y \in [a,b])
\]
with $K_1 \in L^2((a,b)\times(c,d))$, $K_2 \in L^2((c,d)\times(a,b))$, then the induced integral operator $A$ is trace class on $L^2(a,b)$. This is, because $A$ can then be written
as the product of two Hilbert--Schmidt operators.
\item If $K(x,y)$ and $\partial_y K(x,y)$ are continuous on $[a,b]^2$, then the induced integral operator $A$ is trace class on $L^2(a,b)$. This is, because we can write $A$ by partial
integration in the form
\[
Au(x) = K(x,b) \int_a^b u(y)\,dy - \int_a^b \left( \int_y^b \partial_z K(x,z)\,dz \right) u(y)\,dy
\]
as a sum of a rank one operator and an integral operator that is trace class by the first criterion. In particular, integral operators with smooth kernels are trace class
\cite[p.~345]{MR1892228}.
\item A continuous Hermitian\footnote{An\label{ft:hermitian} $L^2$-kernel $K$ is Hermitian if $K(x,y) = \overline{K(y,x)}$. This property is equivalent to the fact that the induced Hilbert--Schmidt integral operator $A$ is
selfadjoint, $A^*=A$.} kernel $K(x,y)$ on $[a,b]$ that satisfies a Hölder condition in the second argument with exponent $\alpha>1/2$, namely
\[
|K(x,y_1)-K(x,y_2)| \leq C |y_1 - y_2|^\alpha\qquad (x,y_1,y_2 \in [a,b]),
\]
induces an integral operator $A$ that is trace class on $L^2(a,b)$; see \citeasnoun[Thm.~IV.8.2]{MR1744872}.
\item If the continuous kernel $K$ induces a selfadjoint, positive-semidefinite integral operator $A$, then $A$ is trace class \cite[Thm.~IV.8.3]{MR1744872}. The hypothesis on $A$ is fulfilled
for positive-semidefinite kernels $K$, that is, if
\begin{equation}\label{eq:semidef}
\sum_{j,k=1}^n \overline{z_j}z_k K(x_j,x_k)\geq 0
\end{equation}
for any $x_1,\ldots,x_n \in (a,b)$, $z \in \C^n$ and any $n \in \N$ \cite[p.~24]{MR2154153}.
\end{enumerate}

\section{Definition and Properties of Fredholm and Operator Determinants}\label{sect:det}

In this section we give a general operator theoretical definition of infinite dimensional determinants and study their relation to the Fredholm determinant.
For a trace class operator $A \in \mathcal{J}_1(\mathcal{H})$ there are several equivalent constructions that all define one and the same \emph{entire} function
\[
d(z) = \det(I+zA)\qquad(z \in \C);
\]
in fact, each construction has been chosen at least once, in different places of the literature, as the basic definition of the operator determinant:

\smallskip

\begin{enumerate}
\item \citeasnoun[p.~157]{MR0246142} define the determinant by the locally uniformly convergent (infinite) product
\begin{equation}\label{eq:detlidskii}
\det(I+z A) = \prod_{n=1}^{N(A)} (1+z\lambda_n(A)),
\end{equation}
which possesses   zeros exactly at $z_n=-1/\lambda_n(A)$, counting multiplicity.
\item \citeasnoun[p.~115]{MR1130394} define the determinant as follows.
Given any sequence of finite rank operators $A_n$ with $A_n \to A$ converging in trace class norm, the sequence of finite dimensional determinants\footnote{\citeasnoun{MR1744872}
have later extended this idea
to generally define traces and determinants on embedded algebras of compact operators by a continuous extension from the finite dimensional case. Even within this general
theory the trace class operators enjoy a most unique position: it is only for them
that the values of trace and determinant are \emph{independent} of the algebra chosen for their definition. On the contrary, if $A$ is Hilbert--Schmidt but not trace class,
by varying the embedded algebra, the values of the trace $\tr(A)$ can be given \emph{any} complex number and the values of the determinant $\det(I+A)$ are either always zero or can be
made to take
\emph{any} value in the set $\C\setminus\{0\}$ \cite[Chap.~VII]{MR1744872}.}
\begin{equation}\label{eq:findef}
\det\left(I + z A_n\projected{\range(A_n)}\right)
\end{equation}
(which are
polynomials in $z$) converges locally uniform to $\det(I+z A)$, independently of the choice of the sequence $A_n$. The existence of at least one such sequence follows from the
singular value representation (\ref{eq:schmidt}).
\item \citeasnoun[p.~1029]{MR188745} define the determinant by what is often called Plemelj's formula\footnote{\citeasnoun[Eq.~(62)]{Plemelj} had given a corresponding form of the Fredholm determinant for integral operators.
However, it can already be found in \citeasnoun[p.~384]{34.0422.02}.}
\begin{equation}\label{eq:detdunford}
\det(I+z A) = \exp(\tr\log(I+zA)) = \exp\left(-\sum_{n=1}^\infty \frac{(-z)^n}{n} \tr A^n \right),
\end{equation}
which converges for $|z| < 1/|\lambda_1(A)|$ and can analytically be  continued as an entire function to all $z \in \C$.
\item \citeasnoun[p.~347]{MR0088665} and \citeasnoun[p.~254]{MR0482328} define the determinant most elegantly with a little exterior algebra \cite{MR0224623}.
With $\bigwedge^n(A) \in \mathcal{J}_1(\bigwedge^n(\mathcal{H}))$ being the
$n^{\text{th}}$ exterior product of $A$, the power series
\begin{equation}\label{eq:detgrothen}
\det(I+z A) = \sum_{n=0}^\infty z^n \tr \bigwedge\nolimits^n(A)
\end{equation}
converges for all $z \in \C$. Note that $\tr \bigwedge\nolimits^n(A) = \sum_{i_1<\cdots < i_n} \lambda_{i_1}(A)\cdots\lambda_{i_n}(A)$
is just the $n^{\text{th}}$ symmetric function of the eigenvalues of $A$.
\end{enumerate}
\smallskip
\noindent Proofs of the equivalence can be found in \cite[Chap.~2]{MR1744872} and \cite[Chap.~3]{MR2154153}. We will make use of all of them in the course of this paper.
We state two important properties \cite[Thm.~3.5]{MR2154153} of the operator determinant: First its multiplication formula,
\begin{equation}\label{eq:mult}
\det(I + A + B + A B) = \det(I+B)\det(I+A)\qquad (A,B \in \mathcal{J}_1(\mathcal{H})),
\end{equation}
then the characterization of invertibility: $\det(I+A) \neq 0$ if and only if the inverse operator $(I+A)^{-1}$ exists.

\medskip
\paragraph{The matrix case} In Section~\ref{sect:quad} we will study the convergence of finite dimensional determinants to operator determinants in terms of the
power series (\ref{eq:detgrothen}). Therefore, we give this series a more common look and feel for the case of a matrix $A \in \C^{m\times m}$. By evaluating the traces with respect to a Schur basis of $A$
one gets
\[
\tr \bigwedge\nolimits^n(A) = \sum_{i_1<\cdots <i_n} \det(A_{i_p, i_q})_{p,q=1}^n = \frac{1}{n!} \sum_{i_1,\ldots,i_n=1}^m  \det(A_{i_p, i_q})_{p,q=1}^n,
\]
that is, the sum of all $n \times n$ principal minors of $A$.
The yields the \citeasnoun{vonKoch} form of the matrix determinant
\begin{equation}\label{eq:vonKoch}
\det(I+z A) = \sum_{n=0}^\infty \frac{z^n}{n!} \sum_{i_1,\ldots,i_n=1}^m  \det(A_{i_p, i_q})_{p,q=1}^n\qquad (A \in \C^{m\times m}).
\end{equation}
(In fact, the series must terminate at $n=m$ since $\det(I+zA)$ is a polynomial of degree $m$ in~$z$.)
A more elementary proof of this classical formula, by a Taylor expansion of the polynomial $\det(I+z A)$, can be found, e.g., in \citeasnoun[p.~495]{MR17382}.

\medskip
\paragraph{The Fredholm determinant for integral operators with continuous kernel}
Suppose that the continuous kernel $K \in C([a,b]^2)$ induces an integral operator $A$ that is trace class on the Hilbert space $\mathcal{H}=L^2(a,b)$.
Then, the traces of $\bigwedge^n(A)$ evaluate to \cite[Thm.~3.10]{MR2154153}
\[
\tr \bigwedge\nolimits^n(A) = \frac{1}{n!} \int_{(a,b)^n} \det(K(t_p,t_q))_{p,q=1}^n\,dt_1\cdots\,dt_n \qquad (n=0,1,2,\ldots).
\]
The power series representation (\ref{eq:detgrothen}) of the operator determinant is therefore exactly Fredholm's expression (\ref{eq:det1}), that is,
\begin{equation}\label{eq:detfred}
\det(I+zA) = \sum_{n=0}^\infty \frac{z^n}{n!} \int_{(a,b)^n} \det(K(t_p,t_q))_{p,q=1}^n\,dt_1\cdots\,dt_n.
\end{equation}
The similarity with von Koch's formula (\ref{eq:vonKoch}) is striking and, in fact, it was just an analogy in form that had led Fredholm to conceive his expression for the determinant. It is important to note, however,
that the right hand side of (\ref{eq:detfred}) perfectly makes sense for any continuous kernel, independently of whether the corresponding integral operator is trace class or not.

\medskip
\paragraph{The regularized determinant for Hilbert--Schmidt operators}

For a general Hilbert--Schmidt operator we only know the convergence of $\sum_n \lambda(A)^2$ but not of $\sum_n \lambda_n(A)$. Therefore, the product (\ref{eq:detlidskii}),
which is meant to define $\det(I+zA)$,
is not known to converge in general. Instead, \citeasnoun{35.0378.02} and \citeasnoun{Carleman} introduced the entire function\footnote{In fact, using such exponential
``convergence factors'' is a classical technique in complex analysis to construct, by means of infinite products, entire functions from their intended sets of zeros, see \citeasnoun[Sect.~3.6]{MR1989049}. A famous
example is
\[
\frac{1}{\Gamma(z)} = z e^{\gamma z} \prod_{n=1}^\infty \left(1+\frac{z}{n}\right)e^{-z/n},
\]
which corresponds to the eigenvalues $\lambda_n(A) = 1/n$ of a Hilbert--Schmidt operator that is not trace class.}
\[
\det\nolimits_2(I+z A) = \prod_{n=1}^{N(A)} (1 + z \lambda_n(A)) e^{-z\lambda_n(A)}\qquad (A \in \mathcal{J}_2(\mathcal{H})),
\]
which also has the property to possess zeros exactly at $z_n=-1/\lambda_n(A)$, counting multiplicity. Plemelj's formula (\ref{eq:detdunford}) becomes \cite[p.~167]{MR1744872}
\[
\det\nolimits_2(I+z A) = \exp\left(-\sum_{n=2}^\infty \frac{(-z)^n}{n} \tr A^n \right)
\]
for $|z|<1/|\lambda_1(A)|$, which perfectly makes sense since $A^2$, $A^3$, \ldots are trace class for $A$ being Hilbert--Schmidt.\footnote{Equivalently one can define \cite[p.~50]{MR2154153}
the regularized determinant by
\[
\det\nolimits_2(I+z A) = \det(I + ((I+zA)e^{-z A} - I)) \qquad (A \in \mathcal{J}_2(\mathcal{H})).
\]
This is because one can then show $(I+zA)e^{-z A} - I \in \mathcal{J}_1(\mathcal{H})$. For integral operators $A$ on $L^2(a,b)$ with an $L^2$ kernel $K$, \citeasnoun[p.~82]{35.0378.02}
had found the equivalent expression
\[
\det\nolimits_2(I+z A) = \sum_{n=0}^\infty \frac{z^n}{n!} \int_{(a,b)^n}
\begin{vmatrix}
0 & K(t_1,t_2) & \cdots & K(t_1,t_n) \\*[1mm]
K(t_2,t_1) & 0 & \cdots & K(t_2,t_n)\\*[1mm]
\vdots & \vdots & \ddots & \vdots\\*[1mm]
K(t_n,t_1) & K(t_n,t_2) & \cdots & 0
\end{vmatrix}\,dt_1\cdots\,dt_n,
\]
simply replacing the problematic ``diagonal'' entries $K(t_j,t_j)$ in Fredholm's determinant (\ref{eq:detfred}) by zero. \citeasnoun[Thm.~9.4]{MR2154153} gives an elegant proof of Hilbert's formula.}
Note that for trace class operators we have
\[
\det(I+zA) = \det\nolimits_2(I+z A) \exp(z\cdot\tr A)\qquad (A \in \mathcal{J}_1(\mathcal{H})).
\]
For integral operators $A$ of the form (\ref{eq:intop}) with a continuous kernel $K$ on $\mathcal{H}=L^2(a,b)$  the Fredholm determinant (\ref{eq:det1}) is related to the Hilbert--Carleman
determinant by the equation \cite[p.~250]{MR0390680}
\[
d(z) = \det\nolimits_2(I+z A) \exp(z \int_a^b K(x,x)\,dx)
\]
in general, even if $A$ is not of trace class. It is important to note, though, that if $A \not\in  \mathcal{J}_1(\mathcal{H})$ with  such a kernel,
we have $\int_a^b K(x,x)\,dx \neq \tr(A)$ simply because  $\tr(A)$ is not well defined by~(\ref{eq:tracedef}) anymore then.

\section{Perturbation Bounds}\label{sect:cond}
In studying the conditioning of operator and matrix determinants we rely on the fundamental perturbation estimate
\begin{equation}\label{eq:perturb0}
|\det(I+A) - \det(I+B)| \leq \|A-B\|_{\scriptscriptstyle\mathcal{J}_1} \exp\left(1+\max(\|A\|_{\scriptscriptstyle\mathcal{J}_1},\|B\|_{\scriptscriptstyle\mathcal{J}_1})\right)
\end{equation}
for trace class operators, which can beautifully be proven by means of complex analysis \cite[p.~45]{MR2154153}. This estimate can be put to the form
\begin{equation}\label{eq:perturb1}
|\det(I+(A+E)) - \det(I+A)| \leq e^{1+\|A\|_{\scriptscriptstyle\mathcal{J}_1}} \cdot \|E\|_{\scriptscriptstyle\mathcal{J}_1} + O(\|E\|_{\scriptscriptstyle\mathcal{J}_1}^2),
\end{equation}
showing that the condition number $\kappa_{\text{abs}}$ of the operator determinant $\det(I+A)$, with respect to absolute errors measured in trace class norm, is bounded by
\[
\kappa_{\text{abs}} \leq e^{1+\|A\|_{\scriptscriptstyle\mathcal{J}_1}}.
\]
This bound can considerably be  improved for certain selfadjoint operators that will play an important role in Section~\ref{sect:random}.

\medskip

\begin{lemma}\label{lem:perturb}
Let $A\in \mathcal{J}_1(\mathcal{H})$ be selfadjoint, positive-semidefinite with $\lambda_1(A) < 1$. If $\|E\|_{\scriptscriptstyle\mathcal{J}_1} < \|(I-A)^{-1}\|^{-1}$ then
\begin{equation}\label{eq:perturb2}
|\det(I-(A+E)) - \det(I-A)| \leq \|E\|_{\scriptscriptstyle\mathcal{J}_1}.
\end{equation}
That is, the condition number $\kappa_{\text{abs}}$ of the determinant $\det(I-A)$, with respect to absolute errors measured in trace class norm, is bounded by $\kappa_{\text{abs}} \leq 1$.
\end{lemma}

\medskip

\begin{proof}
Because of $1>\lambda_1(A)\geq \lambda_2(A) \geq \cdots \geq 0$ there exists the inverse operator $(I-A)^{-1}$. The product formula~(\ref{eq:detlidskii})
 implies $\det(I-A) > 0$, the multiplicativity~(\ref{eq:mult}) of the determinant gives
\[
\det(I-(A+E)) = \det(I-A) \det(I-(I-A)^{-1}E).
\]
Upon applying Plemelj's formula (\ref{eq:detdunford}) and the estimates (\ref{eq:ideal}) and (\ref{eq:sum1lambda}) we get
\begin{multline*}
|\log\det(I-(I-A)^{-1}E)| = \left| \tr\left( \sum_{n=1}^\infty \frac{((I-A)^{-1} E)^n}{n}\right) \right|\\*[2mm]
\leq \sum_{n=1}^\infty \frac{\|(I-A)^{-1}\|^n \cdot\|E\|_{\scriptscriptstyle\mathcal{J}_1}^n}{n} = \log\left(\frac{1}{1-\|(I-A)^{-1}\|\cdot \|E\|_{\scriptscriptstyle\mathcal{J}_1}}\right)
\end{multline*}
if $\|(I-A)^{-1}\|\cdot \|E\|_{\scriptscriptstyle\mathcal{J}_1} < 1$. Hence, exponentiation yields
\begin{multline*}
1-\|(I-A)^{-1}\|\cdot \|E\|_{\scriptscriptstyle\mathcal{J}_1} \leq \det(I-(I-A)^{-1}E)\\*[2mm]
 \leq \frac{1}{1-\|(I-A)^{-1}\|\cdot \|E\|_{\scriptscriptstyle\mathcal{J}_1}} \leq 1+\|(I-A)^{-1}\|\cdot \|E\|_{\scriptscriptstyle\mathcal{J}_1},
\end{multline*}
that is
\[
|\det(I-(A+E)) - \det(I-A)| \leq \det(I-A)\cdot\|(I-A)^{-1}\|\cdot\|E\|_{\scriptscriptstyle\mathcal{J}_1}.
\]
Now, by the spectral theorem for bounded selfadjoint operators we have
\[
\|(I-A)^{-1}\| = \frac{1}{1-\lambda_1(A)} \leq \prod_{n=1}^{N(A)} \frac{1}{1-\lambda_n(A)} = \frac{1}{\det(I-A)}
\]
and therefore $\det(I-A)\cdot\|(I-A)^{-1}\| \leq 1$, which finally proves the assertion.
\end{proof}

\medskip

Thus, for the operators that satisfy the assumptions of this lemma the determinant is a really \emph{well} conditioned quantity---with respect to absolute errors,
like the eigenvalues
of a Hermitian matrix \cite[p.~396]{MR1417720}.

\medskip
\paragraph{Implications on the accuracy of numerical methods}

The Nyström-type method of Section~\ref{sect:quad} requires the calculation of the determinant $\det(I+ A)$ of some matrix $A \in \C^{m\times m}$.
In the presence of roundoff errors, a backward stable method like Gaussian elimination with partial pivoting \cite[Sect.~14.6]{MR1927606} gives
a result that is \emph{exact} for some matrix $\tilde{A} = A + E$ with
\begin{equation}\label{eq:roundoff}
|E_{j,k}| \leq \epsilon |A_{j,k}|\qquad (j,k=1,\ldots,m)
\end{equation}
where $\epsilon$ is a small multiple of the unit roundoff error of the floating-point arithmetic used. We now use the perturbation bounds of this section to
estimate the resulting error in the value of determinant. Since the trace class norm is not a monotone matrix norm \cite[Def.~6.1]{MR1927606}, we cannot make direct use of the
componentwise estimate~(\ref{eq:roundoff}). Instead, we majorize the trace class norm of $m\times m$ matrices $A$ by the Hilbert--Schmidt (Frobenius) norm, which is monotone, using
\[
\|A\|_{\scriptscriptstyle\mathcal{J}_1} \leq \sqrt{m} \|A\|_{\scriptscriptstyle\mathcal{J}_2},\qquad \|A\|_{\scriptscriptstyle\mathcal{J}_2} = \left(\sum_{j,k=1}^m |A_{j,k}|^2\right)^{1/2}.
\]
Thus, the general perturbation bound~(\ref{eq:perturb1}) yields the following a priori estimate of the roundoff error affecting the value of the determinant:
\[
|\det(I+\tilde{A}) - \det(I+A)| \leq \sqrt{m}\|A\|_{\scriptscriptstyle\mathcal{J}_2} \exp\left(1+\|A\|_{\scriptscriptstyle\mathcal{J}_1}\right) \cdot \epsilon + O(\epsilon^2).
\]
If the matrix $A$ satisfies the assumptions of Lemma~\ref{lem:perturb}, the perturbation bound~(\ref{eq:perturb2}) gives the even sharper estimate
\begin{equation}\label{eq:roundoff2}
|\det(I-\tilde{A}) - \det(I-A)| \leq \sqrt{m}\|A\|_{\scriptscriptstyle\mathcal{J}_2}\cdot \epsilon.
\end{equation}
Therefore, if $\det(I-A) \ll \|A\|_{\scriptscriptstyle\mathcal{J}_2}$, we have to be prepared that we probably cannot compute the determinant to the full precision given by the computer arithmetic used.
Some digits will
be lost. A conservative estimate stemming from (\ref{eq:roundoff2}) predicts the loss of at most
\[
\log_{10}\left( \frac{\sqrt{m}\cdot \|A\|_{\scriptscriptstyle\mathcal{J}_2}}{\det(I-A)} \right)
\]
decimal places. For instance, this will affect the \emph{tails} of the probability distributions to be calculated in Section~\ref{sect:random}.

\section{Projection Methods}\label{sect:proj}

The general idea (\ref{eq:findef}) of defining the infinite dimensional determinant $\det(I+A)$ for a trace class operator $A$ by a continuous extension from the finite dimensional case
immediately leads to the concept of a projection method of Galerkin type. We consider a sequence of $m$-dimensional
subspaces $V_m \subset \mathcal{H}$ together with their corresponding orthonormal projections
\[
P_m : \mathcal{H} \to V_m.
\]
The Galerkin projection $P_m AP_m$ of the trace class operator $A$ is of finite rank. Given an orthonormal basis $\phi_1,\ldots,\phi_m$ of $V_m$, its determinant can be effectively
calculated
as the finite dimensional expression
\begin{equation}\label{eq:galerkin}
\det(I+z\,P_mAP_m) = \det\left(I + z\,P_mAP_m\projected{V_m}\right) = \det\left(\delta_{ij} + z\,\langle \phi_i, A \phi_j \rangle\right)_{i,j=1}^m
\end{equation}
if the matrix elements $\langle \phi_i, A \phi_j \rangle$ are numerically accessible.

Because of $\|P_m\|\leq 1$, and
thus $\|P_m AP_m \|_{\scriptscriptstyle\mathcal{J}_1} \leq \| A\|_{\scriptscriptstyle\mathcal{J}_1} \leq 1$,
the perturbation bound~(\ref{eq:perturb0}) gives the simple error estimate
\begin{equation}\label{eq:galerkinerr}
|\det(I+z\,P_mAP_m) - \det(I+z \,A)| \leq \| P_mAP_m - A\|_{\scriptscriptstyle\mathcal{J}_1} \cdot |z|\, e^{1+ |z|\cdot\| A\|_{\scriptscriptstyle\mathcal{J}_1}}.
\end{equation}
For the method to be convergent we therefore have to show that $P_m AP_m \to A$ in {\em trace class norm}. By a general result about the approximation of trace class operators  \cite[Thm.~4.3]{MR1130394}
all we need to know is that $P_m$ converges \emph{pointwise}\footnote{In infinite dimensions $P_m$ cannot converge {\em in norm} since the identity operator is not compact.} to the identity operator $I$. This pointwise
convergence is obviously equivalent to the \emph{consistency} of the family of subspaces $V_m$, that is,
\begin{equation}\label{eq:consist}
\bigcup_{m=1}^\infty V_m \text{ is a dense subspace of $\mathcal{H}$}.
\end{equation}
In summary, we have proven the following theorem.

\smallskip

\begin{theorem}\label{thm:galerkin} Let $A$ be a trace class operator. If the sequence of subspaces satisfies the consistency condition (\ref{eq:consist}), the corresponding Galerkin approximation (\ref{eq:galerkin})
of the operator determinant converges,
\[
\det(I+z\,P_mAP_m) \to \det(I+z\,A)  \qquad (m\to\infty),
\]
uniformly for bounded $z$.
\end{theorem}

\smallskip
A quantitative estimate of the error, that is, in view of (\ref{eq:galerkinerr}), of the projection error $\| P_mAP_m - A\|_{\scriptscriptstyle\mathcal{J}_1}$ in trace class norm, can be based on the
singular value representation~(\ref{eq:schmidt}) of $A$ and its finite-rank truncation $A_N$: (We assume that $A$ is non-degenerate, that is, $N(|A|)=\infty$; since otherwise we could simplify the following by putting $A_N=A$.)
\[
A = \sum_{n=1}^{\infty} s_n(A) \langle u_n,\cdot\,\rangle v_n, \qquad A_N = \sum_{n=1}^{N} s_n(A) \langle u_n,\cdot\,\rangle v_n.
\]
We obtain, by using $\|P_m\|\leq 1$ once more,
\begin{align}\label{eq:projest}
&\| P_mAP_m - A\|_{\scriptscriptstyle\mathcal{J}_1} \leq \| P_mAP_m - P_mA_NP_m\|_{\scriptscriptstyle\mathcal{J}_1} + \| P_mA_NP_m - A_N\|_{\scriptscriptstyle\mathcal{J}_1} + \| A_N- A\|_{\scriptscriptstyle\mathcal{J}_1}\notag\\*[3mm]
&\qquad \qquad\leq 2\| A_N- A\|_{\scriptscriptstyle\mathcal{J}_1} +  \| P_mA_NP_m - A_N\|_{\scriptscriptstyle\mathcal{J}_1}\notag\\*[1mm]
&\qquad \qquad\leq 2 \sum_{n=N+1}^{\infty} s_n(A) \;+\; \sum_{n=1}^N s_n(A) \| \langle P_m u_n,\cdot\,\rangle P_m v_n - \langle u_n,\cdot\,\rangle v_n\|_{\scriptscriptstyle\mathcal{J}_1}\notag\\*[1mm]
&\qquad \qquad\leq 2 \sum_{n=N+1}^{\infty} s_n(A) \;+\; \sum_{n=1}^N s_n(A) \left(\|u_n - P_m u_n\| + \|v_n - P_m v_n\|\right).
\end{align}
There are two competing effects that contribute to making this error bound small: First, there is the convergence of the truncated series of singular values,
\[
 \sum_{n=N+1}^{\infty} s_n(A) \to 0 \qquad (N\to \infty),
\]
which is \emph{independent} of $m$. Second, there is, for \emph{fixed} $N$, the collective approximation
\[
P_m u_n \to u_n,\qquad P_m v_n \to v_n\qquad (m \to \infty)
\]
of the first $N$ singular functions $u_n, v_n$ $(n=1,\ldots,N)$. For instance, given $\epsilon > 0$, we can first choose $N$ large enough to push the first error term in (\ref{eq:projest})
below $\epsilon/2$. After fixing such an $N$, the second error term can be pushed below $\epsilon/2$ for $m$ large enough. This way we have proven Theorem~\ref{thm:galerkin} once more.
However, in general the convergence of the second term might considerably slow down  for growing~$N$. Therefore, a good quantitative bound requires a carefully balanced choice
of $N$ (depending on $m$), which in turn requires some detailed knowledge about the decay of the singular values on the one hand and of the growth of the derivatives of the singular functions on the other
hand (see the example at the end of this section). While some general results are available in the literature for the singular values---e.g., for integral operators $A$ induced by a kernel $K \in C^{k-1,1}([a,b]^2)$ the bound
\begin{equation}\label{eq:snAsmithies}
s_n(A) = O(n^{-k-\frac{1}{2}}) \qquad (n \to \infty)
\end{equation}
obtained by \citeasnoun[Thm.~12]{Smithies37}---the results are sparse for the singular functions \cite[§8.10]{Fenyo}.
Since the quadrature method in Section~\ref{sect:quad} does not require any such knowledge, we refrain from stating a general result and content ourselves with
the case that there is no projection error in the singular functions; that is, we consider projection methods of Ritz--Galerkin type for selfadjoint operators.

\smallskip

\begin{theorem}\label{thm:ritzerr} Let $A$ be a selfadjoint integral operator that is induced by a continuous Hermitian kernel $K$ and that is trace class on the Hilbert space $\mathcal{H} = L^2(a,b)$. Assume that $A$ is not of finite rank and
let $(u_n)$ be an orthonormal basis of eigenfunctions of $A$. We consider the associated Ritz projection $P_m$, that is, the orthonormal projection
\[
P_m : \mathcal{H} \to V_m = \spn\{u_1,\ldots,u_m\}.
\]
Note that in this case
\[
\det(I+z\,P_mAP_m) = \prod_{n=1}^m (1+z\lambda_n(A)).
\]
If $K \in C^{k-1,1}([a,b]^2)$, then there holds the error estimate (\ref{eq:galerkinerr}) with
\[
\| P_mAP_m - A\|_{\scriptscriptstyle\mathcal{J}_1} = o(m^{\frac{1}{2}-k})\qquad (m\to\infty).
\]
If $K$ is bounded analytic on $\mathcal{E}_\rho \times \mathcal{E}_\rho$ (with the ellipse $\mathcal{E}_\rho$ defined in Theorem~\ref{thm:quaderr}),
then the error estimate improves to
\[
\| P_mAP_m - A\|_{\scriptscriptstyle\mathcal{J}_1} = O(\rho^{-m(1-\epsilon)/4})\qquad (m\to\infty),
\]
for any fixed choice of $\epsilon>0$.
\end{theorem}

\smallskip

\begin{proof} With the spectral decompositions
\[
A = \sum_{n=1}^{\infty} \lambda_n(A) \langle u_n,\cdot\,\rangle u_n, \qquad P_m A P_m = A_m = \sum_{n=1}^{m} \lambda_n(A) \langle u_n,\cdot\,\rangle u_n,
\]
at hand the bound (\ref{eq:projest}) simplifies, for $N=m$, to
\[
\| P_mAP_m - A\|_{\scriptscriptstyle\mathcal{J}_1} = \sum_{n=m+1}^\infty |\lambda_n(A)|.
\]
Now, by some results of \citeasnoun[Thm.~7.2 and 10.1]{Hille31}, we have, for $K \in C^{k-1,1}([a,b]^2)$, the eigenvalue decay
\begin{equation}\label{eq:hille}
\lambda_n(A) = o(n^{-k-\frac{1}{2}})\qquad (n\to \infty)
\end{equation}
(which is just slightly stronger than Smithies' singular value bound (\ref{eq:snAsmithies})) and, for $K$ bounded analytic on $\mathcal{E}_\rho \times \mathcal{E}_\rho$,
\[
\lambda_n(A) = O(\rho^{-n(1-\epsilon)/4}) \qquad (n\to \infty);
\]
which proves both assertions.
\end{proof}

\smallskip

However, for kernels of low regularity, by taking into account the specifics of a particular example one often gets better results than stated in this theorem. (An example with an analytic
kernel, enjoying the excellent convergence rates of the second part this theorem, can be found in Section~\ref{sect:random}.)

\smallskip

\paragraph{An example: Poisson's equation} For a given $f \in L^2(0,1)$ the Poisson equation
\[
-u''(x) = f(x),\qquad u(0)=u(1)=0,
\]
with Dirichlet boundary conditions is solved \cite[p.~5]{MR0390680} by the application of the integral operator $A$,
\begin{equation}\label{eq:greenA}
u(x) = A f(x) = \int_0^1 K(x,y) f(y)\,dy,
\end{equation}
which is induced by the Green's  kernel
\begin{equation}\label{eq:greenK}
K(x,y) = \begin{cases}
x(1-y) & x \leq y, \\*[1mm]
y(1-x) & \text{otherwise}.
\end{cases}
\end{equation}
Since $K$ is Lipschitz continuous, Hermitian, and positive definite, we know from the results of Section~\ref{sect:trace} that $A$ is a selfadjoint trace class operator on $\mathcal{H}=L^2(0,1)$.
The eigenvalues and
normalized eigenfunctions of $A$ are those of the Poisson equation which are known to be
\[
\lambda_n(A) = \frac{1}{\pi^2 n^2},\qquad u_n(x) = \sqrt{2}\sin(n \pi x)\qquad (n=1,2,\ldots).
\]
Note that the actual decay of the eigenvalues is better than the general Hille--Ta\-mar\-kin bound (\ref{eq:hille}) which would, because
of $K \in C^{0,1}([0,1]^2)$, just give $\lambda_n(A)=o(n^{-3/2})$.
The trace formulas (\ref{eq:tracedef}) and~(\ref{eq:tracediag}) reproduce a classical formula of Euler's, namely
\[
\sum_{n=1}^\infty \frac{1}{\pi^2 n^2} = \tr(A) = \int_0^1 K(x,x)\,dx = \int_0^1 x(1-x)\,dx = \frac{1}{6};
\]
whereas, by (\ref{eq:detlidskii}) and the product representation of the sine function, the Fredholm determinant explicitly evaluates to the entire function
\begin{equation}\label{eq:greendet}
\det(I-z A) = \prod_{n=1}^\infty \left(1-\frac{z}{\pi^2 n^2}\right) = \frac{\sin(\sqrt z)}{\sqrt z}.
\end{equation}
The sharper perturbation bound of Lemma~\ref{lem:perturb} applies and we get, for each finite dimensional subspace $V_m \subset \mathcal{H}$ and the corresponding orthonormal projection
$P_m: \mathcal{H} \to V_m$,
the error estimate
\begin{equation}\label{eq:greenerr}
|\det(I-P_m A P_m) - \det(I-A)| \leq \| P_mAP_m - A\|_{\scriptscriptstyle\mathcal{J}_1}.
\end{equation}
Now, we study two particular families of subspaces.

\smallskip
\paragraph{Trigonometric polynomials}
Here, we consider the subspaces
\[
V_m = \spn \{ \sin(n\pi\, \cdot) : n=1,\ldots,m\} = \spn\{ u_n : n=1,\ldots, m\},
\]
which are exactly those spanned by the eigenfunctions of $A$. In this case, the projection method is of Ritz--Galerkin type;
the estimates become particularly simple since we have the explicit spectral decomposition
\[
P_m A P_m - A = \sum_{n=m+1}^\infty \lambda_n(A) \langle u_n,\cdot\,\rangle u_n
\]
of the error operator.
Hence, the error bound (\ref{eq:greenerr}) directly evaluates to
\begin{multline}\label{eq:ritzbound}
|\det(I-P_m A P_m) - \det(I-A)| \\*[2mm]
\leq \| P_mAP_m - A\|_{\scriptscriptstyle\mathcal{J}_1} = \sum_{n=m+1}^\infty \lambda_n(A) = \frac{1}{\pi^2} \sum_{n=m+1}^\infty \frac{1}{n^2} \leq \frac{1}{\pi^2 m}.
\end{multline}
Figure~\ref{fig:ritz} shows that this upper bound overestimates the error in the Fredholm determinant by just about 20\%.

\begin{figure}[tbp]
\begin{center}
\begin{minipage}{0.75\textwidth}
{\includegraphics[width=\textwidth]{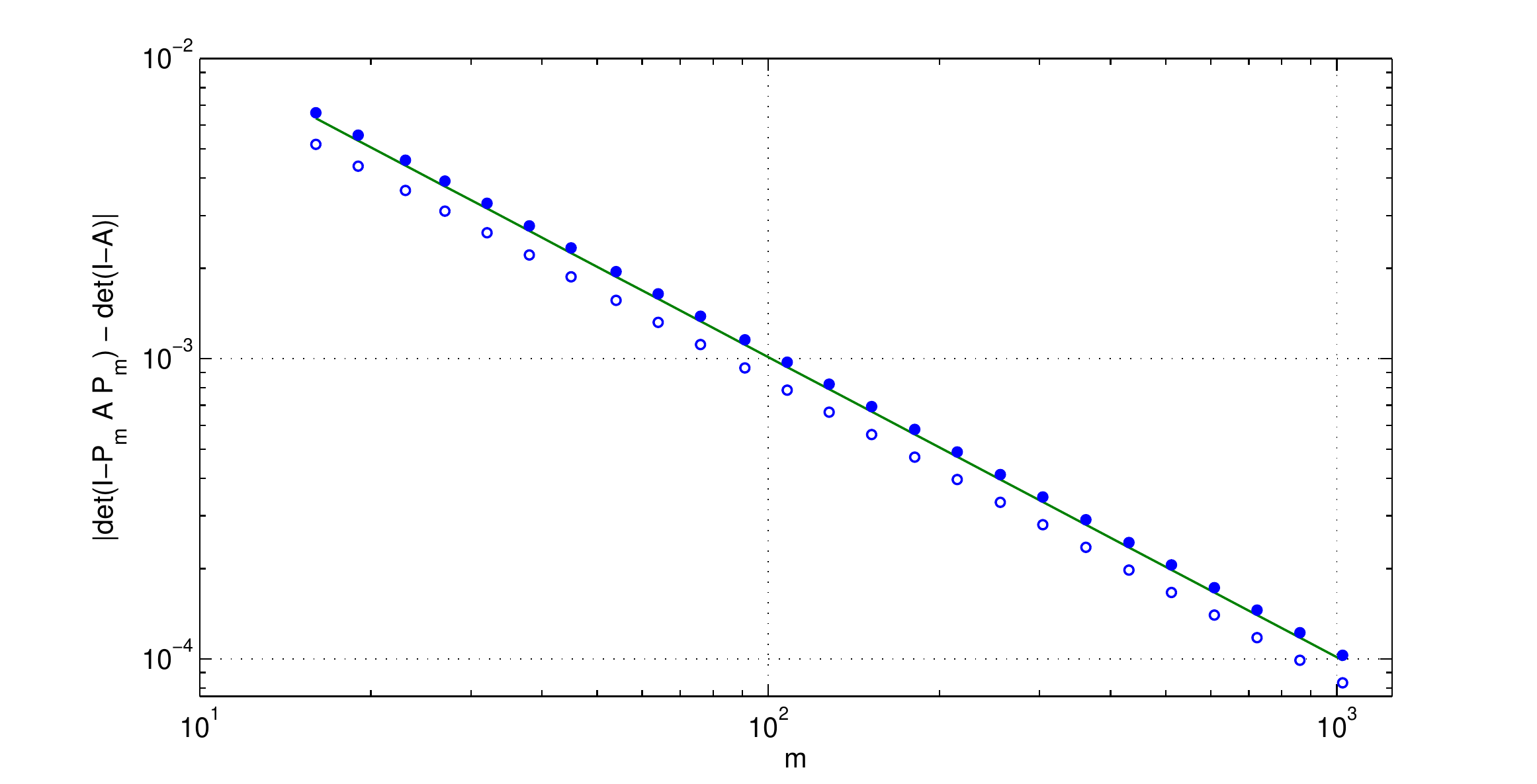}}
\end{minipage}
\end{center}\vspace*{-0.0625cm}
\caption{Convergence of Ritz--Galerkin (circles) and Galerkin (dots) for approximating the Fredholm determinant of the integral operator induced by Green's kernel of Poisson's equation.
The solid line shows the upper bound $1/\pi^2 m$ of Ritz--Galerkin as given in (\ref{eq:ritzbound}).}
\label{fig:ritz}
\end{figure}

\smallskip
\paragraph{Algebraic polynomials} Here, we consider the subspaces of algebraic polynomials of order $m$, that is,
\[
V_m = \{u \in L^2(0,1) : \text{$u$ is a polynomial of degree at most $m-1$}\}.
\]
We apply the bound given in (\ref{eq:projest}) and obtain (keeping in mind that $A$ is selfadjoint)
\[
\| P_mAP_m - A\|_{\scriptscriptstyle\mathcal{J}_1} \leq 2 \sum_{n=N+1}^{\infty} \lambda_n(A) \;+\; 2\sum_{n=1}^N \lambda_n(A)\,\|u_n - P_m u_n\|
\]
with a truncation index $N$ yet to be skilfully chosen.
As before in (\ref{eq:ritzbound}), the first term of this error bound can be estimated by $2/\pi^2N$.
For the second term we recall that the projection error $\|u_n - P_m u_n\|$ is, in fact, just the error of polynomial best approximation of the eigenfunction
$u_n$ with respect to the $L^2$-norm.
A standard Jackson-type inequality \cite[p.~219]{MR1261635} from approximation theory teaches us that
\[
\|u_n - P_m u_n\| \leq \frac{c_k}{m^k}\|u^{(k)}_n\| =  c_k \frac{(\pi n)^k}{m^k},
\]
where $c_k$ denotes a constant that depends on the smoothness level $k$. A \emph{fixed} eigenfunction $u_n$ (being an entire function in fact)
is therefore approximated beyond every algebraic order in the dimension $m$; but with increasingly larger constants for higher ``wave numbers'' $n$.
We thus get, with some further constant $\tilde{c}_k$ depending on $k \geq 2$,
\[
\| P_m A P_m - A\|_{\scriptscriptstyle\mathcal{J}_1} \leq \tilde{c}_k \left( \frac{1}{N} + \frac{N^{k-1}}{(k-1)m^k} \right).
\]
We now balance the two error terms by minimizing this bound: the optimal truncation index $N$ turns out to be exactly $N=m$, which finally yields the estimate
\[
|\det(I-P_m A P_m) - \det(I-A)| \leq \| P_m A P_m - A\|_{\scriptscriptstyle\mathcal{J}_1} \leq \frac{\tilde{c}_k}{1-k^{-1}}\, m^{-1}.
\]
Thus, in contrast to the approximation of a {\em single} eigenfunction, for the Fredholm determinant the \emph{order} of the convergence estimate does finally not depend on
the choice of $k$ anymore; we obtain the same $O(m^{-1})$ behavior as for the Ritz--Galerkin method. In fact, a concrete numerical calculation\footnote{By (\ref{eq:galerkin})
all we need to know for implementing the Galerkin method are the matrix elements $\langle \phi_i,A\phi_j\rangle$ for the normalized orthogonal polynomials $\phi_n$
(that is, properly rescaled Legendre polynomials) on the interval~$[0,1]$.
A~somewhat lengthy but straightforward calculation shows that these elements are given by
\[
(\langle \phi_i,A \phi_j\rangle)_{i,j} = \begin{pmatrix}
\frac{1}{12} & 0 & b_0 \\*[2mm]
0 & \frac{1}{60} & 0 & b_1  \\
b_0 & 0 & a_1 & \ddots & \ddots \\
& b_1 & \ddots & a_2  \\
& & \ddots & & \ddots
\end{pmatrix}
\]
with the coefficients
\[
a_n = \frac{1}{2(2n+1)(2n+5)},\qquad b_n = -\frac{1}{4(2n+3)\sqrt{(2n+1)(2n+5)}}.
\]}
shows that this error estimate really reflects the actual order of the error decay of the Galerkin method, see Figure~\ref{fig:ritz}.

\smallskip

\paragraph{Remark} The analysis of this section has shown that the error decay of the projection methods is essentially determined by
the decay
\[
\sum_{k=m+1}^\infty s_k(A) \to 0
\]
of the singular values of $A$, which in turn is related to the smoothness of the kernel $K$ of the integral operator $A$.
In the next section, the error analysis of Nyström-type quadrature
methods will relate in a much more direct fashion to the smoothness of the kernel $K$, giving even much better error bounds, a priori and in actual numerical
computations. For instance, the Green's kernel (\ref{eq:greenA}) of low regularity can be treated by an $m$-dimensional approximation of the determinant
with an actual convergence rate of $O(m^{-2})$ instead of $O(m^{-1})$ as for the projection methods. Moreover, these methods are much simpler and straightforwardly implemented.

\section{Quadrature Methods}\label{sect:quad}

In this section we directly approximate the Fredholm determinant (\ref{eq:det1}) using the Nyström-type method (\ref{eq:ny}) that we have motivated at length in Section~\ref{sect:intro}.
We assume throughout that the kernel $K$ is a continuous function on $[a,b]^2$. The notation simplifies considerably by introducing the $n$-dimensional functions $K_n$ defined by
\[
K_n(t_1,\ldots,t_n) = \det\left(K(t_p,t_q)\right)_{p,q=1}^n.
\]
Their properties are given in Lemma~\ref{lem:Kn} of the appendix. We then write the Fredholm determinant shortly as
\[
d(z) = 1 + \sum_{n=1}^\infty \frac{z^n}{n!} \int_{[a,b]^n} K_n(t_1,\ldots,t_n)\,dt_1\cdots\,dt_n.
\]
For a given quadrature formula
\[
Q(f) = \sum_{j=1}^m w_j f(x_j)\,\approx\,\int_a^b f(x)\,dx
\]
we define the associated Nyström-type approximation of $d(z)$ by the expression
\begin{equation}\label{eq:detny}
d_Q(z) = \det \left(\delta_{ij} + z\,w_i K(x_i,x_j)\right)_{i,j=1}^m.
\end{equation}
The key to error estimates and a convergence proof is the observation that we can rewrite $d_Q(z)$ in a form that
closely resembles the Fredholm determinant. Namely, by using the von Koch form~(\ref{eq:vonKoch}) of  matrix determinants,
the multi-linearity of minors, and by introducing the $n$-dimensional product rule~(\ref{eq:productrule}) associated with $Q$ (see the appendix), we
get
\begin{align*}
d_Q(z) &= 1+\sum_{n=1}^\infty \frac{z^n}{n!} \sum_{j_1,\ldots,j_n=1}^m \det\left( w_{j_p} K(x_{j_p},x_{j_q}) \right)_{p,q=1}^n \\*[1mm]
&= 1+\sum_{n=1}^\infty \frac{z^n}{n!} \sum_{j_1,\ldots,j_n=1}^m w_{j_1}\cdots w_{j_n}\,\det\left( K(x_{j_p},x_{j_q}) \right)_{p,q=1}^n\\*[1mm]
&= 1+\sum_{n=1}^\infty \frac{z^n}{n!} \sum_{j_1,\ldots,j_n=1}^m w_{j_1}\cdots w_{j_n}\,K_n(x_{j_1},\ldots,x_{j_n})\\*[1mm]
&= 1+\sum_{n=1}^\infty \frac{z^n}{n!}\, Q^n(K_n).
\end{align*}
Thus, alternatively to the motivation given in the introductory Section~\ref{sect:intro}, we could have introduced the Nyström-type method
by approximating each of the multi-dimensional integrals in the  power series defining the Fredholm determinant
with a product quadrature rule. Using this form, we observe that the error is given by
\begin{equation}\label{eq:nyerr}
d_Q(z) - d(z) = \sum_{n=1}^\infty \frac{z^n}{n!} \left( Q^n(K_n) - \int_{[a,b]^n} K_n(t_1,\ldots,t_n)\,dt_1\cdots\,dt_n \right),
\end{equation}
that is, by the exponentially generating function of the quadrature errors for the functions $K_n$.
The following theorem generalizes a result that \citeasnoun[p.~59]{35.0378.02} had proven for a specific class of quadrature formulae, namely, the rectangular rule.

\smallskip

\begin{theorem}\label{thm:nyconv}
If a family $Q_m$ of quadrature rules converges for continuous functions,
then the corresponding Nyström-type approximation of the Fredholm determinant converges,
\[
d_{Q_m}(z) \to d(z)\qquad (m\to\infty),
\]
uniformly for bounded $z$.
\end{theorem}

\smallskip

\begin{proof} Let $z$ be bounded by $M$ and choose any $\epsilon>0$. We split the series (\ref{eq:nyerr}) at an index $N$ yet to be chosen, getting
\begin{multline*}
|d_{Q_m}(z) - d(z)| \leq \sum_{n=1}^N \frac{M^n}{n!} \left| Q^n_m(K_n) - \int_{[a,b]^n} K_n(t_1,\ldots,t_n)\,dt_1\cdots\,dt_n \right| \\*[2mm]
+ \sum_{n=N+1}^\infty \frac{M^n}{n!} \left| Q^n_m(K_n) - \int_{[a,b]^n} K_n(t_1,\ldots,t_n)\,dt_1\cdots\,dt_n \right|
\end{multline*}
Let $\Lambda$ be the stability bound of the convergent family $Q_m$ of quadrature rules (see Theorem~\ref{thm:polya} of the appendix) and put $\Lambda_1 = \max(\Lambda,b-a)$.
Then, by Lemma~\ref{lem:Kn}, the second part of the splitting can be bounded by
\begin{multline*}
\sum_{n=N+1}^\infty \frac{M^n}{n!} \left| Q^n_m(K_n) - \int_{[a,b]^n} K_n(t_1,\ldots,t_n)\,dt_1\cdots\,dt_n \right| \\*[2mm]
\leq \sum_{n=N+1}^\infty \frac{M^n}{n!} \left( |Q^n_m(K_n)| + |\int_{[a,b]^n} K_n(t_1,\ldots,t_n)\,dt_1\cdots\,dt_n | \right) \\*[2mm]
\leq \sum_{n=N+1}^\infty \frac{M^n}{n!} \left(\Lambda^n + (b-a)^n\right) \|K_n\|_{\scriptscriptstyle L^\infty}
\leq 2 \sum_{n=N+1}^\infty \frac{n^{n/2}}{n!} (M \Lambda_1 \|K\|_{\scriptscriptstyle L^\infty})^n.
\end{multline*}
The last series converges by Lemma~\ref{lem:phi} and the bound can, therefore, be pushed below $\epsilon/2$ by choosing $N$ large enough. After fixing such an $N$, we can certainly
also push the first part of the splitting, that is,
\[
\sum_{n=1}^N \frac{M^n}{n!} \left| Q^n_m(K_n) - \int_{[a,b]^n} K_n(t_1,\ldots,t_n)\,dt_1\cdots\,dt_n \right|\,,
\]
below $\epsilon/2$, now for $m$ large enough, say $m\geq m_0$, using the convergence of the product rules $Q_m^n$ induced by $Q_m$ (see Theorem~\ref{thm:quaderrn}). In summary we get
\[
|d_{Q_m}(z) - d(z)| \leq \epsilon
\]
for all $|z|\leq M$ and $m \geq m_0$, which proves the asserted convergence of the Nyström-type method.
\end{proof}

\smallskip

If the kernel $K$ enjoys, additionally, some smoothness, we can prove error estimates that exhibit, essentially, the same rates of convergence as for the quadrature of the
sections $x \mapsto K(x,y)$ and $y \mapsto K(x,y)$.

\smallskip

\begin{theorem}\label{thm:nyerr}
If $K \in C^{k-1,1}([a,b]^2)$, then for each quadrature rule $Q$ of order $\nu \geq k$ with positive weights there holds the error estimate
\[
|d_Q(z)-d(z)| \leq c_k\, 2^k(b-a)^k \nu^{-k} \,\Phi\!\left(|z|(b-a)\|K\|_k\right),
\]
where $c_k$ is the constant (depending only on $k$) from Theorem~\ref{thm:quaderr}, and $\|K\|_k$ and $\Phi$ are the norm and function defined in (\ref{eq:Knormk}) and (\ref{eq:phi}), respectively.

If $K$ is bounded analytic on $\mathcal{E}_\rho \times \mathcal{E}_\rho$ (with the ellipse $\mathcal{E}_\rho$ defined in Theorem~\ref{thm:quaderr}), then for each quadrature rule $Q$ of
order $\nu$ with positive weights there holds the error estimate
\[
|d_Q(z)-d(z)| \leq \frac{4\,\rho^{-\nu}}{1-\rho^{-1}} \,\Phi\!\left(|z|(b-a)\|K\|_{\scriptscriptstyle L^\infty(\mathcal{E}_\rho \times \mathcal{E}_\rho)}\right).
\]
\end{theorem}

\begin{proof}
By Theorem~\ref{thm:quaderrn} and Lemma~\ref{lem:Kn} we can estimate the error (\ref{eq:nyerr}) in both cases in the form
\[
|d_Q(z)-d(z)| \leq \alpha \sum_{n=1}^\infty \frac{n^{(n+2)/2}}{n!} \,(|z|\beta)^n = \alpha\, \Phi(|z|\beta)\,;
\]
with the particular values $\alpha = c_k\, 2^k(b-a)^k \nu^{-k}$ and $\beta= (b-a)\, \|K\|_k$
in the first case, and $\alpha = 4\,\rho^{-\nu}/(1-\rho^{-1})$ and $\beta = (b-a)\,\|K\|_{\scriptscriptstyle L^\infty(\mathcal{E}_\rho \times \mathcal{E}_\rho)}$
in the second case. This proves both assertions.
\end{proof}

\smallskip

An example with an analytic
kernel, enjoying the excellent convergence rates of the second part this theorem, can be found in Section~\ref{sect:random}.

Note that Theorem~\ref{thm:nyerr} is based on a general result (Theorem~\ref{thm:quaderr}) about quadrature errors that stems from the convergence rates of polynomial best approximation. There are cases (typically of
low regularity), however, for which certain quadrature formulae enjoy convergence rates that are actually \emph{better} than best approximation. The Nyström-type
method inherits this behavior; one would just have to repeat the proof of Theorem~\ref{thm:nyerr} then.
We refrain from stating a general theorem, since this would involve bounds on the highest derivatives involving weights\footnote{For the interval $[-1,1]$ this weight
would be $(1-x^2)^{1/2}$, see \citeasnoun[§4.8.1]{MR760629}.} that take into account the boundary of the interval $[a,b]$. Instead, we content ourselves
with the detailed discussion of a particular example.

\smallskip

\paragraph{An example: Poisson's equation} We revisit the example of Section~\ref{sect:proj}, that is the integral operator (\ref{eq:greenA}) belonging to the Green's kernel $K$
(defined in (\ref{eq:greenK})) of Poisson's equation. Recall from (\ref{eq:greendet}) that
\[
d(-1) = \det(I-A) = \sin(1).
\]
The kernel $K$ is just Lipschitz continuous, that is, $K \in C^{0,1}([0,1]^2)$. If we apply the Nyström-type method with the $m$-point Gauss--Legendre (order $\nu=2m$)
or the Curtis--Clenshaw (order $\nu=m$) formulae as the underlying quadrature rule $Q_m$, Theorem~\ref{thm:nyerr} proves an error bound of the form
\[
d_{Q_m}(-1) - d(-1) = O(m^{-1}),
\]
which superficially indicates the same convergence rate as for the $m$-dimensional Galerkin methods of Section~\ref{sect:proj}.
However, the actual numerical computation shown in Figure~\ref{fig:ny1} exhibits the far better convergence rate of $O(m^{-2})$.
\begin{figure}[tbp]
\begin{center}
\begin{minipage}{0.75\textwidth}
{\includegraphics[width=\textwidth]{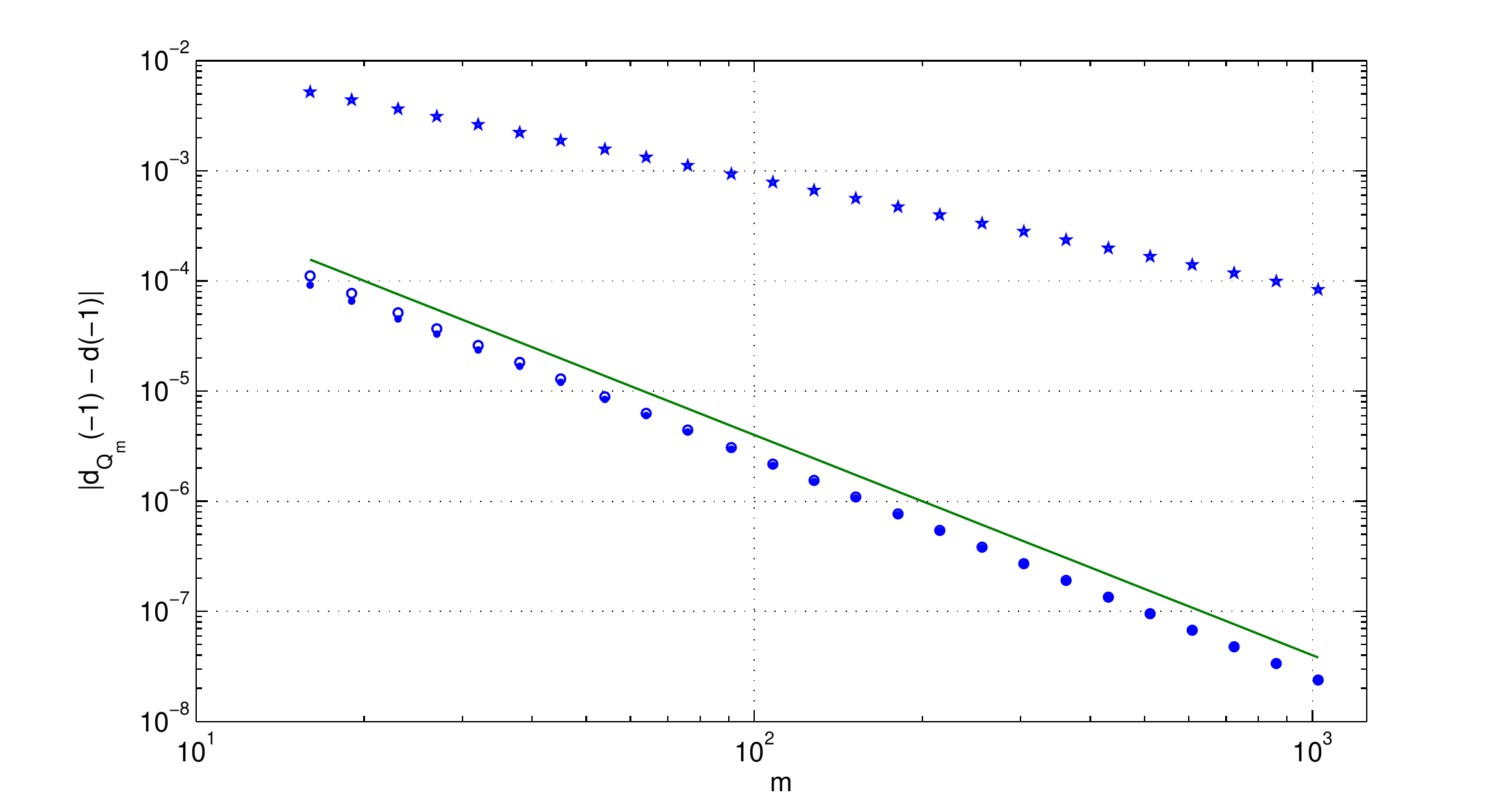}}
\end{minipage}
\end{center}\vspace*{-0.0625cm}
\caption{Convergence of the Nyström-type method for approximating the Fredholm determinant of the integral operator  induced by Green's kernel (\ref{eq:greenK}) of Poisson's equation; the
underlying quadrature rules $Q_m$ are the $m$-point Gauss--Legendre (dots) and Clenshaw--Curtis (circles) rules. Note that, in accordance with \protect\citeasnoun{Tref08}, both behave essentially the same.
The solid line shows the function $1/25 m^2$, just to indicate the rate of convergence.  For comparison we have included the results of the Ritz--Galerkin method (stars) from Figure~\ref{fig:ritz}.}
\label{fig:ny1}
\end{figure}
This deviation can be understood in detail as follows:

On the one hand, by inverse theorems of approximation theory \cite[p.~220]{MR1261635}, valid for \emph{proper} subintervals of $[a,b]$,
the polynomial best approximation (of degree $m$) of sections of the Green's kernel $K$ cannot give a better rate than $O(m^{-1})$; since otherwise those sections
could not show jumps in the first derivative. Given the line of arguments leading from polynomial best approximation to Theorem~\ref{thm:nyerr}, the error estimate of $O(m^{-1})$ was
therefore the
\emph{best} that could be established \emph{this} way.

On the other hand, the sections of the Green's kernel look like piecewise linear hat functions. Therefore, the coefficients $a_m$ of their Chebyshev expansions decay as $O(m^{-2})$
\cite[Eq.~(4.8.1.3)]{MR760629}. Given this decay rate, one can then prove---see, for Gauss--Legendre, \citeasnoun[Eq.~(4.8.1.7)]{MR760629} and, for Clenshaw--Curtis, \citeasnoun[Eq.~(9)]{MR0305555}---%
that the quadrature error is of rate $O(m^{-2})$, too. Now, one can lift this estimate to the Nyström-like method essentially as in Theorem~\ref{thm:nyerr};
thus \emph{proving} in fact that
\[
d_{Q_m}(-1) - d(-1) = O(m^{-2}),
\]
as numerically observed.

\smallskip

\paragraph{Remark} This ``superconvergence'' property of certain quadrature rules, as opposed to best approximation, for kernels with jumps in a higher derivative is therefore
also the deeper reason that the Nyström-type method then outperforms the projection methods of Section~\ref{sect:proj} (see Figure~\ref{fig:ny1}): Best approximation, by
direct (Jackson) and inverse (Bernstein) theorems of approximation
theory, is strongly tied with the regularity of $K$. And this, in turn, is tied to the decay of the singular values of the induced integral operator $A$, which governs the convergence
rates of projection methods.

\paragraph{A note on implementation} If the quadrature weights are positive (which in view of Theorem~\ref{thm:polya} is anyway the better choice), as is the case for Gauss--Legendre and
Clenshaw--Curtis, we recommend to implement the Nyström-type method (\ref{eq:detny}) in the equivalent, symmetric form
\begin{equation}\label{eq:detnysym}
d_Q(z) = \det(I+z A_Q), \qquad A_Q = \left(w_i^{1/2} K(x_i,x_j) w_j^{1/2}\right)_{i,j=1}^m.
\end{equation}
(Accordingly short Matlab and Mathematica code is given in the introductory Section~\ref{sect:intro}.)
The reason is that the $m\times m$-matrix $A_Q$ inherits some important structural properties from the integral operator $A$:

\smallskip

\begin{itemize}
\item If $A$ is selfadjoint, then $A_Q$ is Hermitian (see Footnote~\ref{ft:hermitian}).\\*[-3mm]
\item If $A$ is positive semidefinite, then, by (\ref{eq:semidef}), $A_Q$ is positive semidefinite, too.
\end{itemize}

\smallskip

\noindent
This way, for instance, the computational cost for calculating the finite-dimensional determinant is cut to half, if by structural inheritance $I+zA_Q$ is Hermitian positive definite; the Cholesky decomposition can
then be employed instead of Gaussian elimination with partial pivoting.


\section{Application to Some Entire Kernels of Random Matrix Theory}\label{sect:random}

In this section we study two important examples, stemming from random matrix theory, with entire kernels. By Theorem~\ref{thm:nyerr}, the Nyström-type method based on Gauss--Legendre or Curtis--Clenshaw quadrature
has to exhibit exponential convergence.

\begin{figure}[tbp]
\begin{center}
\begin{minipage}{0.75\textwidth}
{\includegraphics[width=\textwidth]{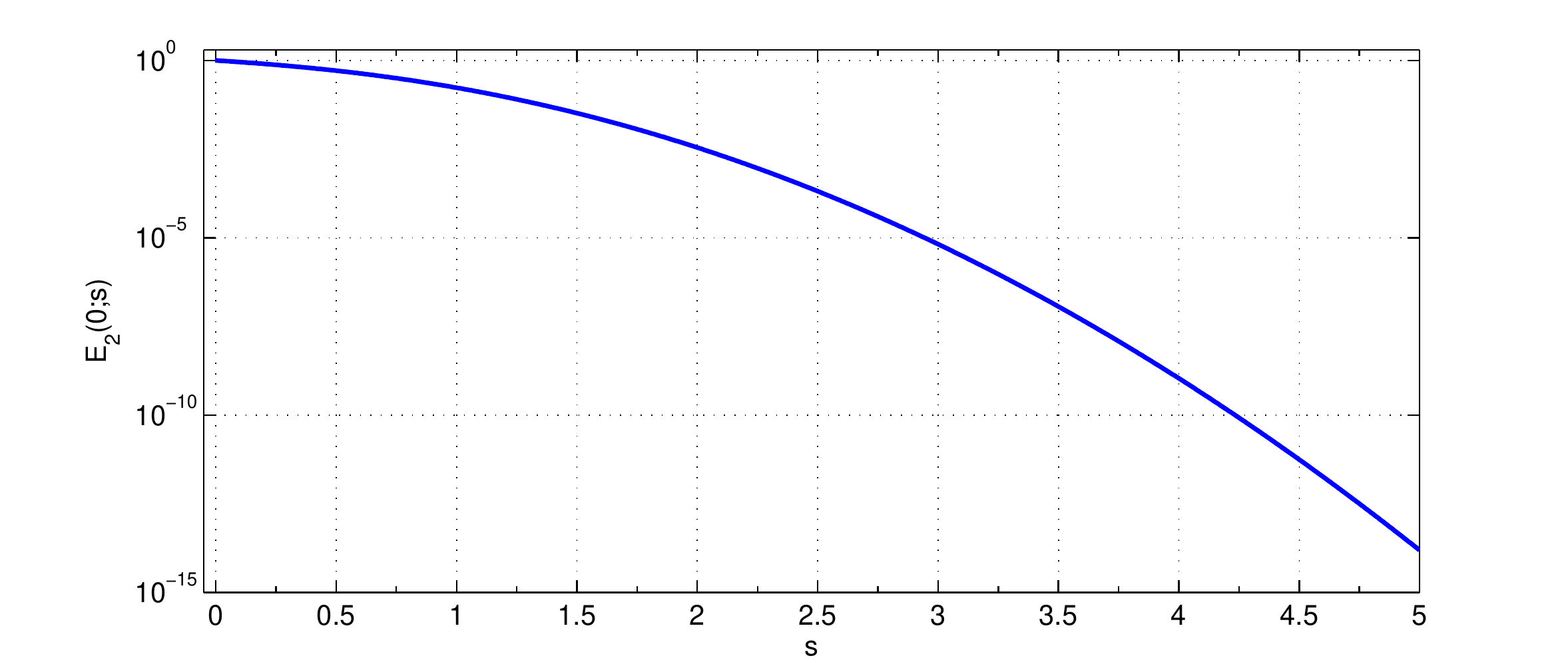}}
\end{minipage}
\end{center}\vspace*{-0.0625cm}
\caption{The probability $E_2(0;s)$ that an interval of length $s$ does not contain, in the bulk scaling limit of level spacing $1$, an eigenvalue of the Gaussian unitary ensemble (GUE). The result shown was
calculated with the Nyström-like method based on Gauss--Legendre with $m=30$; and cross-checked against the asymptotic
expansion $\log E_2(0;s) = -\pi^2 s^2/8 - \log(s)/4 + \log(2)/3 - \log(\pi)/4 + 3\zeta'(-1) + O(s^{-1})$ for $s\to\infty$ \protect\cite{MR1469319}.}
\label{fig:E0s}
\end{figure}

\subsection{The sine kernel}\label{sect:sine}

The probability $E_2(0;s)$ (shown in Figure~\ref{fig:E0s}) that an interval of length $s$ does not contain, in the bulk scaling limit of level spacing $1$, an eigenvalue
of the Gaussian unitary ensemble (GUE) is given \cite[Sect.~6.3]{MR2129906} by the Fredholm determinant
\[
E_2(0;s) = \det\left(I - A_s\right)
\]
of the integral operator $A_s$ on $L^2(0,s)$ that is induced by the sine kernel $K$:
\[
A_s u(x) = \int_0^s K(x,y)u(y)\,dy,\qquad K(x,y) = \frac{\sin(\pi(x-y))}{\pi(x-y)}.
\]
Note that $K(x,y)$ is Hermitian and entire on $\C\times\C$; thus $A_s$ is a selfadjoint operator of trace class on $L^2(0,s)$.
(This is already much more than we would need to know for successfully applying and understanding the Nyström-type method. However, to facilitate a
comparison with the Ritz--Galerkin method, we analyze the operator $A_s$ in some more detail.)  The factorization
\begin{equation}\label{eq:sinkernesqrt}
K(x,y) = \frac{1}{2\pi}\int_{-\pi}^\pi e^{i (x-y)\xi}\, d\xi = \frac{1}{2\pi}\int_{-\pi}^\pi e^{i x\xi}e^{-i y\xi}\, d\xi
\end{equation}
of the kernel implies that $A_s$ is positive \emph{definite} with maximal eigenvalue $\lambda_1(A_s)<1$; since, for $0 \neq u \in L^2(0,s)$, we obtain
\begin{multline*}
0 < \langle u ,A_s u\rangle = \int_{-\pi}^\pi \left| \frac{1}{\sqrt{2\pi}}\int_0^s e^{-i x \xi} u(x) \,dx \right|^2 \,d\xi \\*[2mm]
= \int_{-\pi}^\pi |\hat u(\xi)|^2\,d\xi < \int_{-\infty}^\infty |\hat u(\xi)|^2\,d\xi = \|\hat u\|_{\scriptscriptstyle L^2}^2 = \|u\|_{\scriptscriptstyle  L^2}^2.
\end{multline*}
Here, in stating that the inequalities are \emph{strict}, we have used the fact that the Fourier transform $\hat u$ of the function $u \in L^2(0,s)$, which has compact support, is an \emph{entire} function.
Therefore, the perturbation bound of Lemma~\ref{lem:perturb} applies and we obtain, for Ritz--Galerkin as for any Galerkin method, like in the example of Section~\ref{sect:proj}, the basic error estimate
\[
|\det(I-P_m A_s P_m) - \det(I-A_s)| \leq \| P_mA_sP_m - A_s\|_{\scriptscriptstyle\mathcal{J}_1}.
\]
Now, Theorems~\ref{thm:ritzerr} and \ref{thm:nyerr} predict a rapid, exponentially fast
convergence of the Ritz--Galerkin and the Nyström-type methods: In fact, an $m$-dimensional approximation will give an error that decays like $O(e^{-c m})$, for any fixed $c>0$ since, for entire kernels,
the parameter $\rho>1$ can  be chosen arbitrarily in these theorems.

\smallskip
\paragraph{Details of the implementation of the Ritz--Galerkin method} There is certainly no general recipe on how to actually construct the
Ritz--Galerkin method for a specific example, since one would have
to know, more or less exactly, the eigenvalues of $A$. In the case of the sine kernel, however, \citeasnoun{Gau61} had succeeded in doing so.
(See also \citeasnoun[p.~411]{MR1659828} and  \citeasnoun[pp.~124--126]{MR2129906}.) He had observed that the integral operator $\tilde A_t$ on $L^2(-1,1)$, defined by
\[
\tilde A_t u(x) = \int_{-1}^1 e^{i\pi t x y} u(y)\,dy
\]
(which is, by (\ref{eq:sinkernesqrt}), basically a rescaled ``square-root'' of $A_{2t}$), is commuting with the selfadjoint, second-order differential operator
\[
L u(x) = \frac{d}{dx}\left((x^2-1) u'(x) \right) + \pi^2t^2 x^2 u(x)
\]
with boundary conditions
\[
(1-x^2) u(x) |_{x=\pm 1} = (1-x^2) u'(x) |_{x = \pm 1} = 0.
\]
Thus, both operators share the same set of eigenfunctions $u_n$, namely the radial prolate spheroidal wave functions (using the notation of Mathematica 6.0)
\[
u_n(x) = S^{(1)}_{n,0}(\pi t,x) \qquad (n=0,1,2\ldots).
\]
These special functions are even for $n$ even, and odd for $n$ odd. By plugging them into the integral operator $\tilde A_t$ \citename{Gau61} had obtained, after evaluating at $x=0$, the eigenvalues
\[
\lambda_{2k}(\tilde A_t) = \frac{1}{u_{2k}(0)}\int_{-1}^1 u_{2k}(y)\,dy,\qquad \lambda_{2k+1}(\tilde A_t) = \frac{i\pi t}{u_{2k+1}'(0)}\int_{-1}^1 u_{2k+1}(y) y \,dy.
\]
Finally, we have (starting with the index $n=0$ here)
\[
\lambda_n(A_s) = \frac{s}{4} |\lambda_n(\tilde A_{s/2})|^2\qquad (n=0,1,2,\ldots).
\]
Hence, the $m$-dimensional Ritz--Galerkin approximation of $\det(I-A_s)$ is given by
\[
\det(I-P_m A_s P_m) = \prod_{n=0}^{m-1} (1-\lambda_n(A_s)).
\]
While \citename{Gau61} himself had to rely on tables of the spheroidal wave functions \cite{MR0074130}, we can use the fairly recent implementation of these special functions by \citeasnoun{MR1985521}, which now comes
with Mathematica 6.0. This way, we get the following implementation:

\begin{center}\label{prog:mathspher}
\includegraphics[width=0.75\textwidth]{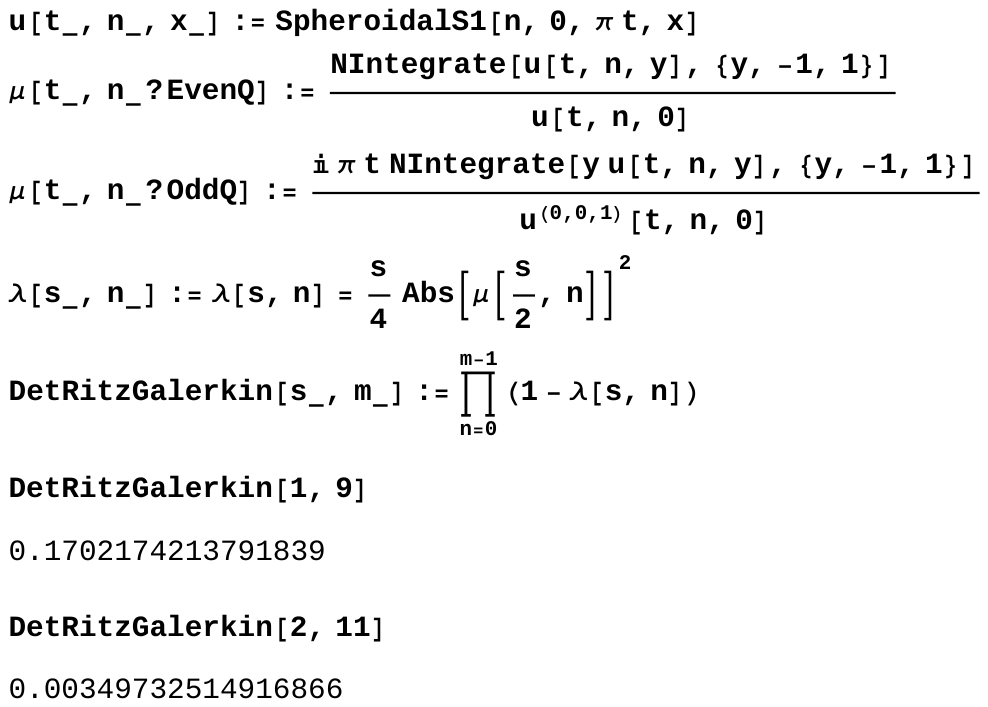}
\end{center}

\begin{figure}[tbp]
\begin{center}
\begin{minipage}{0.49\textwidth}
{\includegraphics[width=\textwidth]{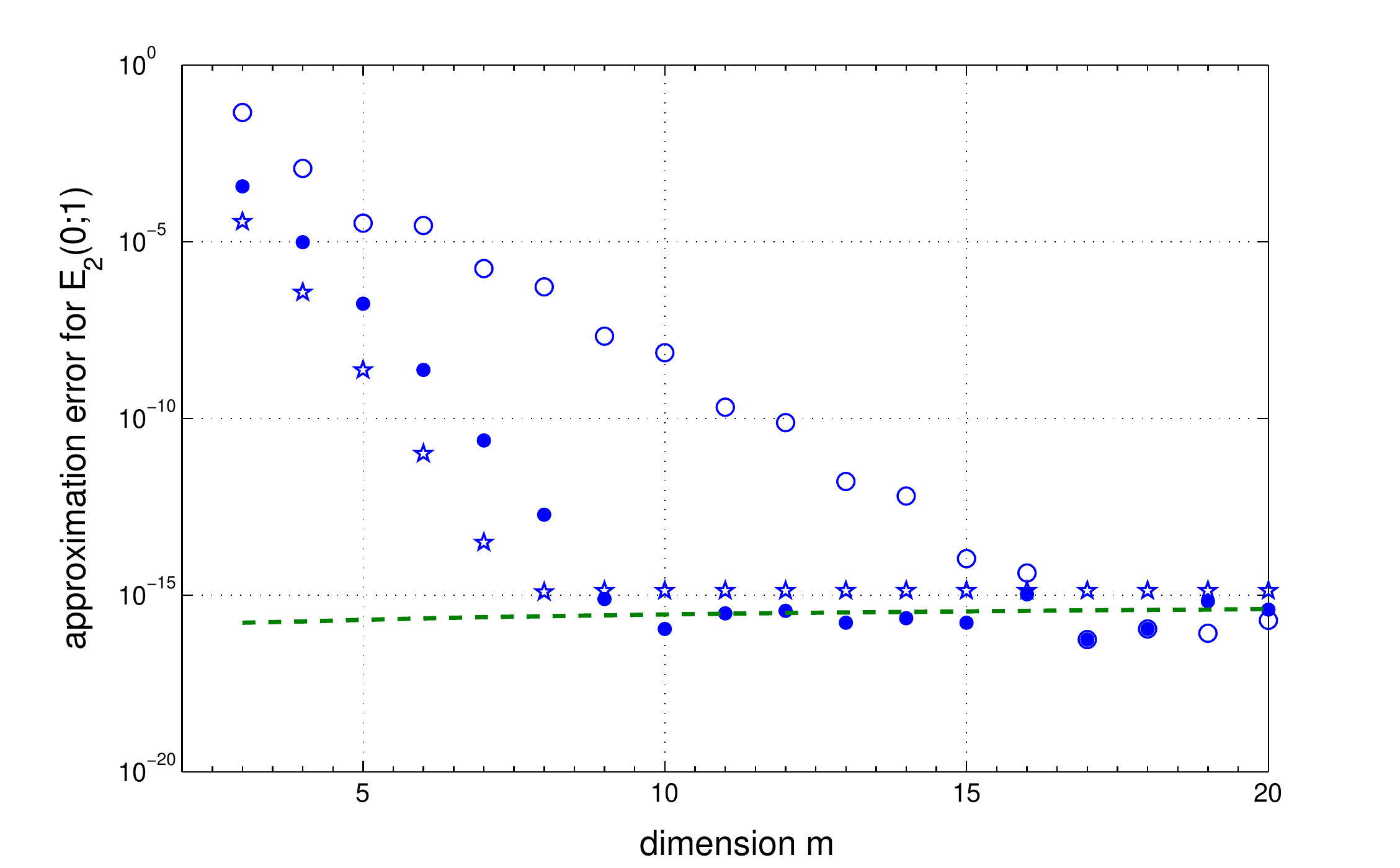}}
\end{minipage}
\hfil
\begin{minipage}{0.49\textwidth}
{\includegraphics[width=\textwidth]{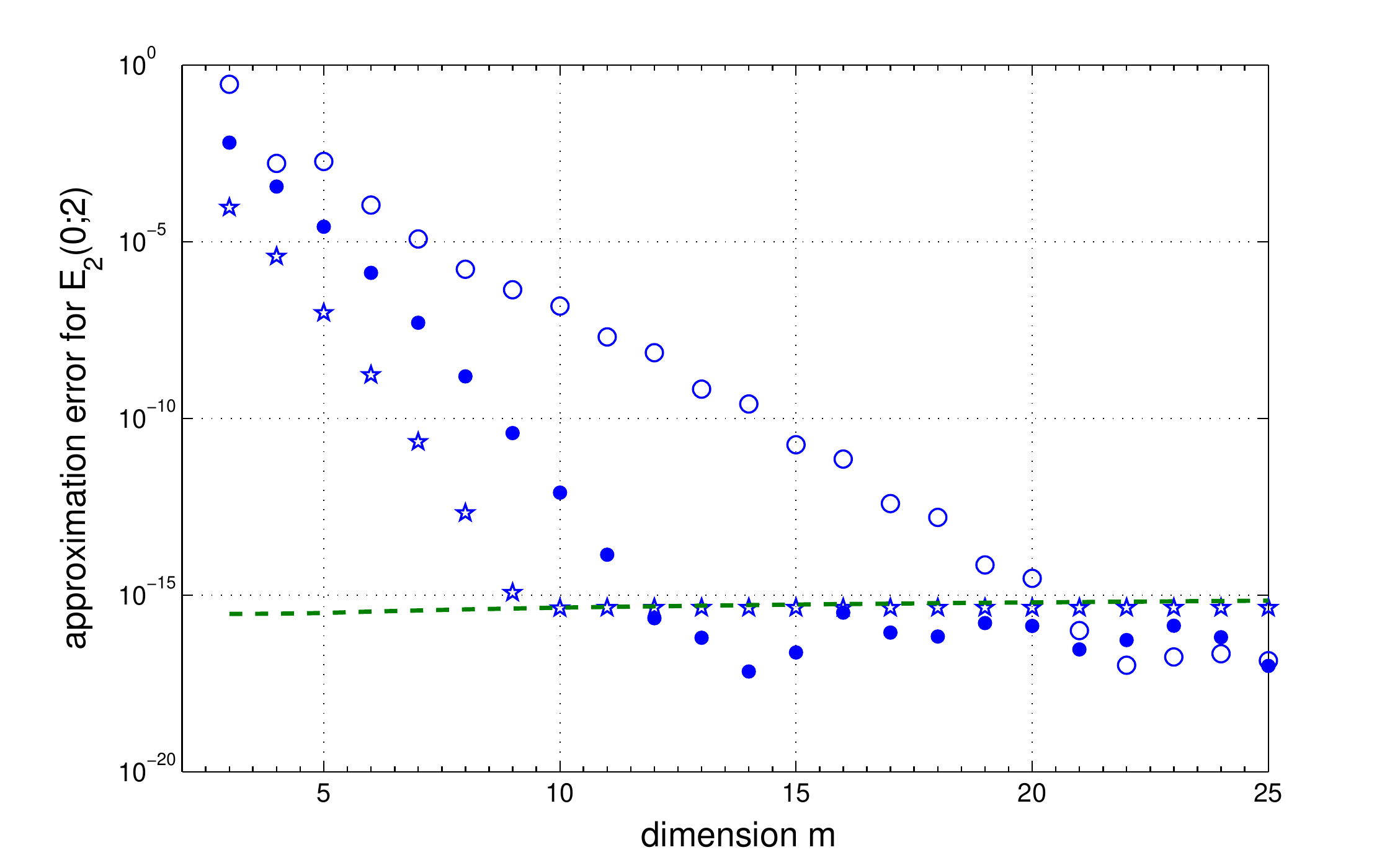}}
\end{minipage}
\end{center}\vspace*{-0.0625cm}
\caption{Convergence of various $m$-dimensional approximations of the Fredholm determinants $E_2(0;1)$ (left) and $E_2(0;2)$ (right): the Nyström-type quadrature methods based on Gauss--Legendre (dots) and Curtis--Clenshaw (circles), as
well as \protect\possessivecite{Gau61} Ritz--Galerkin method based on spheroidal wave functions (stars).
The dashed line shows the amount, according to (\ref{eq:roundoff2}), of roundoff error due to the numerical evaluation of the finite-dimensional determinants; all calculations were done in IEEE double arithmetic with
machine precision $\epsilon = 2^{-53}$.}
\label{fig:E0err}
\end{figure}

\noindent
Given all this, one can understand that \possessivecite{MR573370} beautiful discovery of expressing $E_2(0;s)$ by formula (\ref{eq:jimbo}) in terms of the fifth Painlevé transcendent
was generally considered to be a major break-through even for its numerical evaluation \cite[Footnote~10]{MR1791893}. However,  note how much less knowledge suffices for
the application of the far more general Nyström-type method: continuity of $K$ makes it applicable, and $K$ being entire guarantees rapid, exponentially fast convergence. That is all.

\smallskip
\paragraph{An actual numerical experiment} Figure~\ref{fig:E0err} shows the convergence (in IEEE machine arithmetic) of an actual  calculation of the numerical values $E_2(0;1)$ and $E_2(0;2)$.
We observe
that the Nyström-type method based on Gauss--Legendre has an exponentially fast convergence rate comparable  to the Ritz--Galerkin method. Clenshaw--Curtis needs
a dimension $m$
that is about twice as large as for Gauss--Legendre to achieve the same accuracy. This matches the fact that Clenshaw--Curtis has the order $\nu=m$, which is half
the order $\nu=2m$
of Gauss--Legendre,
and shows that the bounds of Theorem~\ref{thm:nyerr} are rather sharp with respect to $\nu$ (there is no ``kink'' phenomenon here, cf. \citeasnoun[p.~84]{Tref08}).
The dashed line shows the amount, as estimated in (\ref{eq:roundoff2}), of roundoff error that stems from the numerical evaluation of the finite dimensional $m\times m$-determinant itself.
Note that this bound is essentially the same for all the three methods and can easily be calculated in course of the numerical evaluation. We observe that this
bound is explaining the actual onset of numerical ``noise'' in all the three methods reasonably well.

\smallskip
\paragraph{Remark}
Note that the Nyström-type method \emph{outperforms} the Ritz--Galer\-kin me\-thod by far. First, the Nyström-type method is general, simple, and straightforwardly implemented (see the code
given in Section~\ref{sect:intro}); in contrast, the Ritz--Galerkin
depends on some detailed knowledge about the eigenvalues and requires numerical access to the spheroidal wave functions. Second, there is no substantial gain,
as compared to the Gauss--Legendre based method,
in the convergence rate from knowing the eigenvalues exactly. Third, and most important, the computing time
for the Ritz--Galerkin runs well into several minutes, whereas both versions of the Nyström-type method require just a few \emph{milliseconds}.

\subsection{The Airy kernel}\label{sect:airy}

\begin{figure}[tbp]
\begin{center}
\begin{minipage}{0.75\textwidth}
{\includegraphics[width=\textwidth]{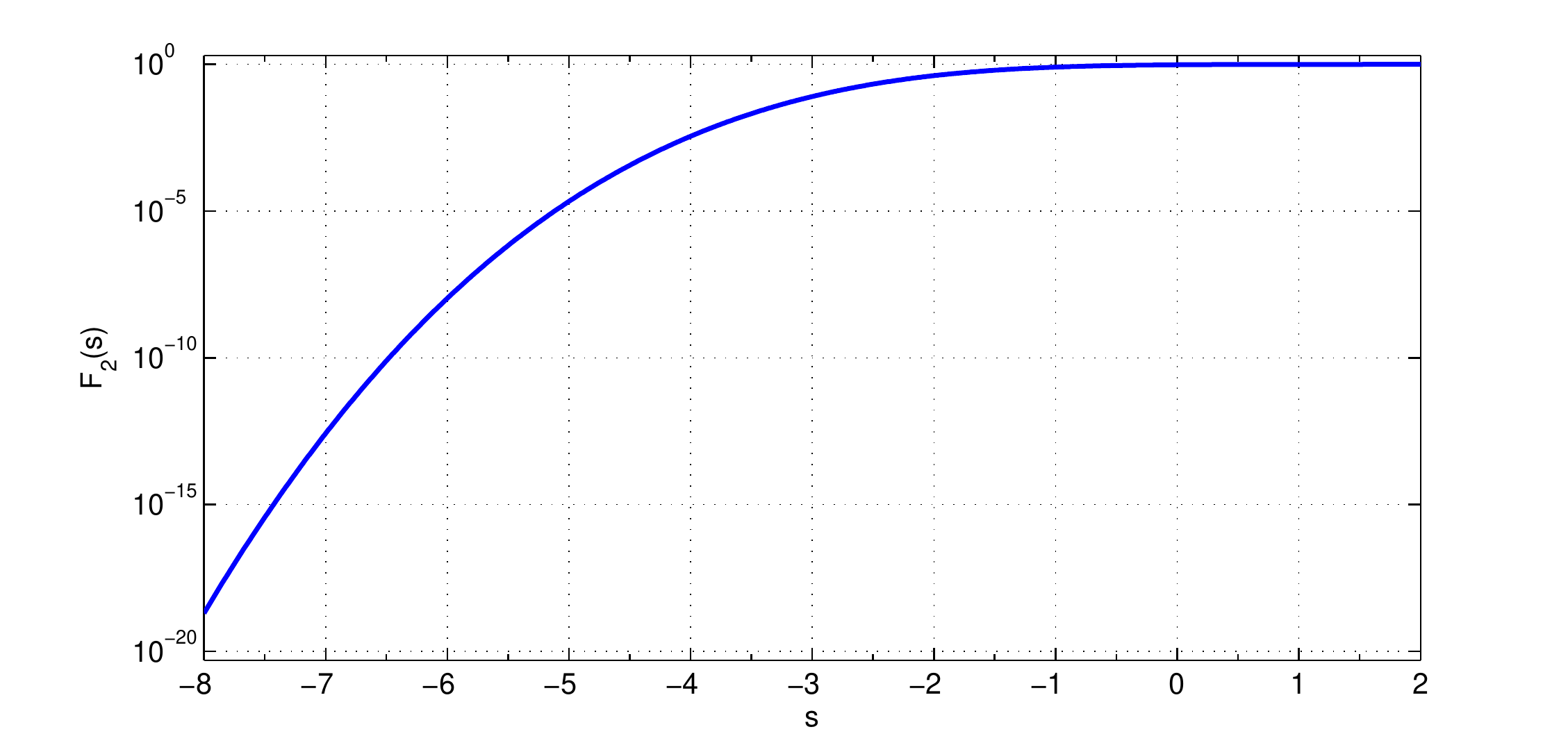}}
\end{minipage}
\end{center}\vspace*{-0.0625cm}
\caption{The Tracy--Widom distribution $F_2(s)$; that is, in the edge scaling limit, the probability of the maximal eigenvalue of the Gaussian unitary ensemble (GUE) being not larger than $s$. The result shown
was calculated with the Nyström-like method based on Gauss--Legendre with $m=50$; and cross-checked against the asymptotic
expansion $\log F_2(-s) = -s^3/12 - \log(s)/8 + \log(2)/24 + \zeta'(-1) + O(s^{-3/2})$ for $s\to\infty$ \protect\cite{MR2373439}.}
\label{fig:F2}
\end{figure}

The Tracy--Widom distribution $F_2(s)$ (shown in Figure~\ref{fig:F2}), that is, in the edge scaling limit, the probability of the maximal eigenvalue of the Gaussian unitary ensemble (GUE) being not larger than $s$, is given \cite[§24.2]{MR2129906}
by the Fredholm determinant
\begin{equation}\label{eq:F2}
F_2(s) = \det(I-A_s)
\end{equation}
of the integral operator $A_s$ on $L^2(s,\infty)$ that is induced by the Airy kernel $K$:
\begin{equation}\label{eq:airykernel}
A_s u(x) = \int_s^\infty K(x,y) u(y)\,dy,\qquad K(x,y) = \frac{\Ai(x)\Ai'(y) - \Ai(y)\Ai'(x)}{x-y}.
\end{equation}
Note that $K$ is Hermitian and entire on $\C \times \C$; thus $A_s$ is selfadjoint. (Again, this is already already about all we would need to know for successfully applying and understanding the Nyström-type method.
However, we would like to show that, as for the sine kernel, the strong perturbation bound of Lemma~\ref{lem:perturb} applies to the Airy kernel, too.)
There is the factorization \cite[Eq.~(4.5)]{MR1257246}
\[
K(x,y) = \int_0^\infty \Ai(x+\xi)\Ai(y+\xi)\,d\xi,
\]
which relates the Airy kernel with the Airy transform \cite[§4.2]{MR2114198} in a similar way as the sine kernel is related by (\ref{eq:sinkernesqrt}) with the Fourier transform. This proves, because of the super-exponential
decay of $\Ai(x) \to 0$ as $x\to 0$, that $A_s$ is the product of two Hilbert--Schmidt operators on $L^2(s,\infty)$ and therefore of trace class. Moreover, $A_s$ is positive semi-definite with maximal eigenvalue $\lambda_1(A) \leq 1$;
since by the Parseval--Plancherel equality \cite[Eq.~(4.27)]{MR2114198} of the Airy transform we obtain, for $u \in L^2(s,\infty)$,
\begin{multline*}
0 \leq \langle u,A_s u\rangle = \int_0^\infty \left| \int_0^\infty \Ai(x+\xi)u(x)\,dx \right|^2 \,d\xi \\*[2mm]
\leq  \int_{-\infty}^\infty \left| \int_0^\infty \Ai(x+\xi)u(x)\,dx \right|^2 \,d\xi = \|u\|_{L^2}^2.
\end{multline*}
More refined analytic arguments, or a carefully controlled numerical approximation, show the strict inequality $\lambda_1(A) < 1$; the perturbation bound of Lemma~\ref{lem:perturb} applies.

\smallskip

\paragraph{Modification of the Nyström-type method for infinite intervals} The quadrature methods of Section~\ref{sect:quad} are not directly applicable here, since the integral operator $A_s$ is defined
by an integral over the infinite interval $(s,\infty)$. We have the following three options:

\begin{figure}[tbp]
\begin{center}
\begin{minipage}{0.725\textwidth}
{\includegraphics[width=\textwidth]{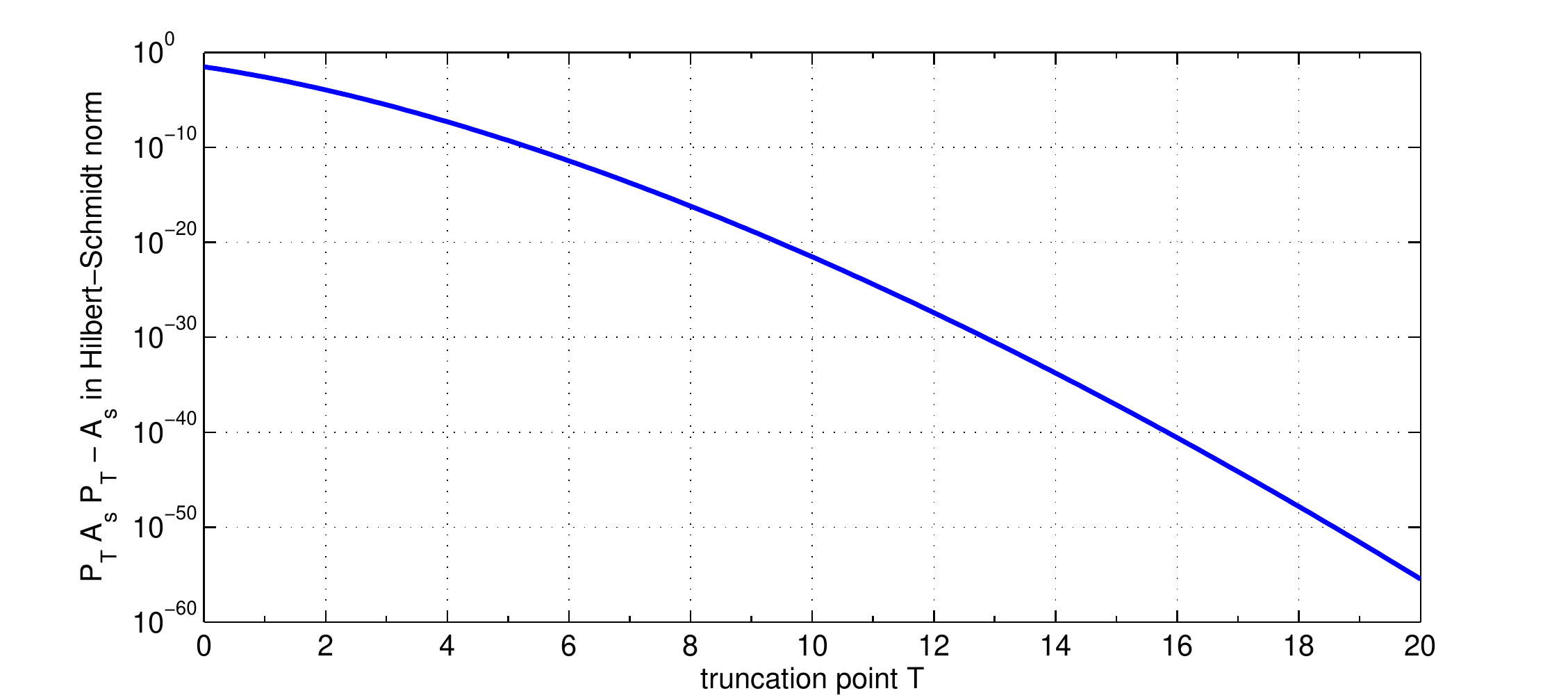}}
\end{minipage}
\end{center}\vspace*{-0.0625cm}
\caption{Values of the expression $\|P_T A_s P_T - A_s \|_{\scriptscriptstyle\mathcal{J}_2}$, which bounds, by (\ref{eq:truncerr}), the error in $F_2(s)$ committed by truncating the integral in (\ref{eq:airykernel}) at a point $T>s$.}
\label{fig:truncerr}
\end{figure}

\smallskip

\begin{enumerate}
\item Using a Gauss-type quadrature formula on $(s,\infty)$ that is tailor-made for the super-exponential decay of the Airy function. Such formulae have recently been constructed by \citeasnoun{MR1896388}.\\*[-3mm]
\item Truncating the integral in (\ref{eq:airykernel}) at some point $T>s$. That is, before using the Nyström-type method with a quadrature formula on the finite interval $[s,T]$ (for which the
second part of Theorem~\ref{thm:nyerr} is then applicable, showing exponential convergence), we approximate the
Fredholm determinant (\ref{eq:F2}) by
\[
\det(I-P_T A_s P_T) = \det\left(I - A_s\projected{L^2(s,T)}\right),
\]
where the orthonormal projection $P_T : L^2(s,\infty) \to L^2(s,T)$, $P u = u \cdot \chi_{[s,T]}$, denotes the multiplication operator by the characteristic function of $[s,T]$. This way
we commit an additional truncation error, which has, by passing through the perturbation bound of Lemma~\ref{lem:perturb}, the computable bound
\begin{multline}\label{eq:truncerr}
|\det(I-P_T A_s P_T) - \det(I-A_s) | \leq \|P_T A_s P_T - A_s \|_{\scriptscriptstyle\mathcal{J}_1} \leq \\*[2mm]
 \|P_T A_s P_T - A_s \|_{\scriptscriptstyle\mathcal{J}_2} = \left(\int_T^\infty\int_T^\infty |K(x,y)|^2\,dxdy\right)^{1/2}.
\end{multline}
Figure~\ref{fig:truncerr} shows this bound as a function of the truncation point $T$. We observe that, for the purpose of calculating (within IEEE machine arithmetic) $F_2(s)$ for $s \in [-8,2]$---as shown in Figure~\ref{fig:F2}---,
a truncation point at $T=16$ would be more than sufficiently safe.\\*[-3mm]
\item Transforming the infinite intervals to finite ones. By using a monotone and smooth transformation $\phi_s:(0,1) \to (s,\infty)$, defining the transformed integral operator $\tilde A_s$ on $L^2(0,1)$ by
\[
\tilde A_s u(\xi) = \int_0^1 \tilde K_s(\xi,\eta) u(\eta)\,d\eta,\quad \tilde K_s(\xi,\eta) = \sqrt{\phi'_s(\xi)\phi'_s(\eta)}\, K(\phi_s(\xi),\phi_s(\eta)),
\]
gives the identity
\[
F_s(s) = \det\left(I-A_s\projected{L^2(s,\infty)}\right) = \det\left(I-\tilde A_s\projected{L^2(0,1)}\right).
\]
For the super-exponentially decaying Airy kernel $K$ we suggest the transformation
\begin{equation}\label{eq:transform}
\phi_s(\xi) = s + 10 \tan(\pi \xi/2)\qquad (\xi\in(0,1)).
\end{equation}
Note that though $\tilde K(\xi,\eta)$ is a smooth function on $[0,1]$ it possesses, as a function on $\C\times \C$, essential singularities on the lines $\xi=1$ or $\eta=1$. Hence,
we can only apply the first part of Theorem~\ref{thm:nyerr} here, which then shows, for Gauss--Legendre and Clenshaw--Curtis, a super-algebraic convergence rate, that is, $O(m^{-k})$ for arbitrarily high algebraic order $k$.
The actual numerical experiments reported in Figure~\ref{fig:F2err} show, in fact, even exponential convergence.
\end{enumerate}

\smallskip

\begin{figure}[tbp]
\begin{center}
\begin{minipage}{0.49\textwidth}
{\includegraphics[width=\textwidth]{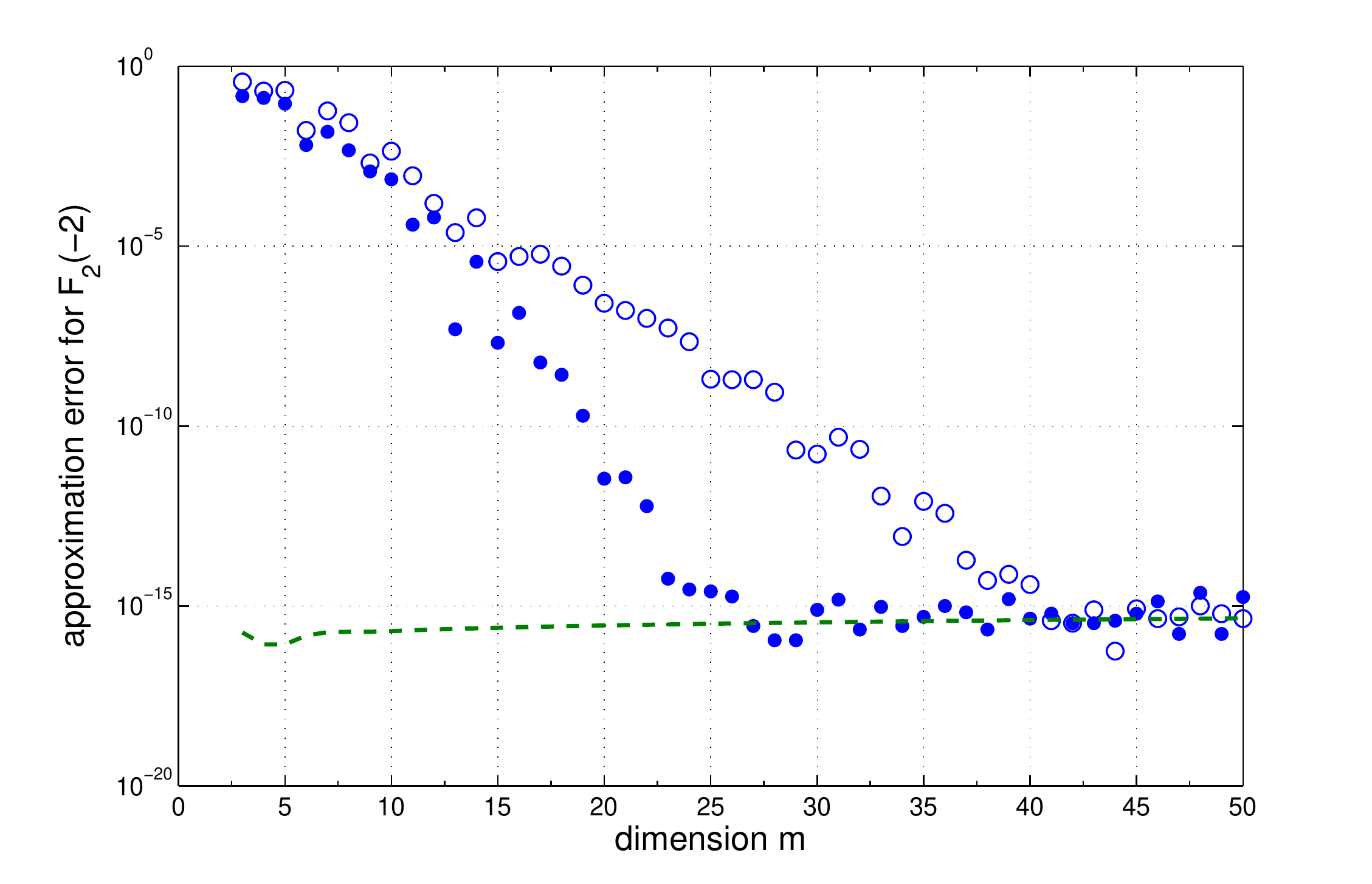}}
\end{minipage}
\hfil
\begin{minipage}{0.49\textwidth}
{\includegraphics[width=\textwidth]{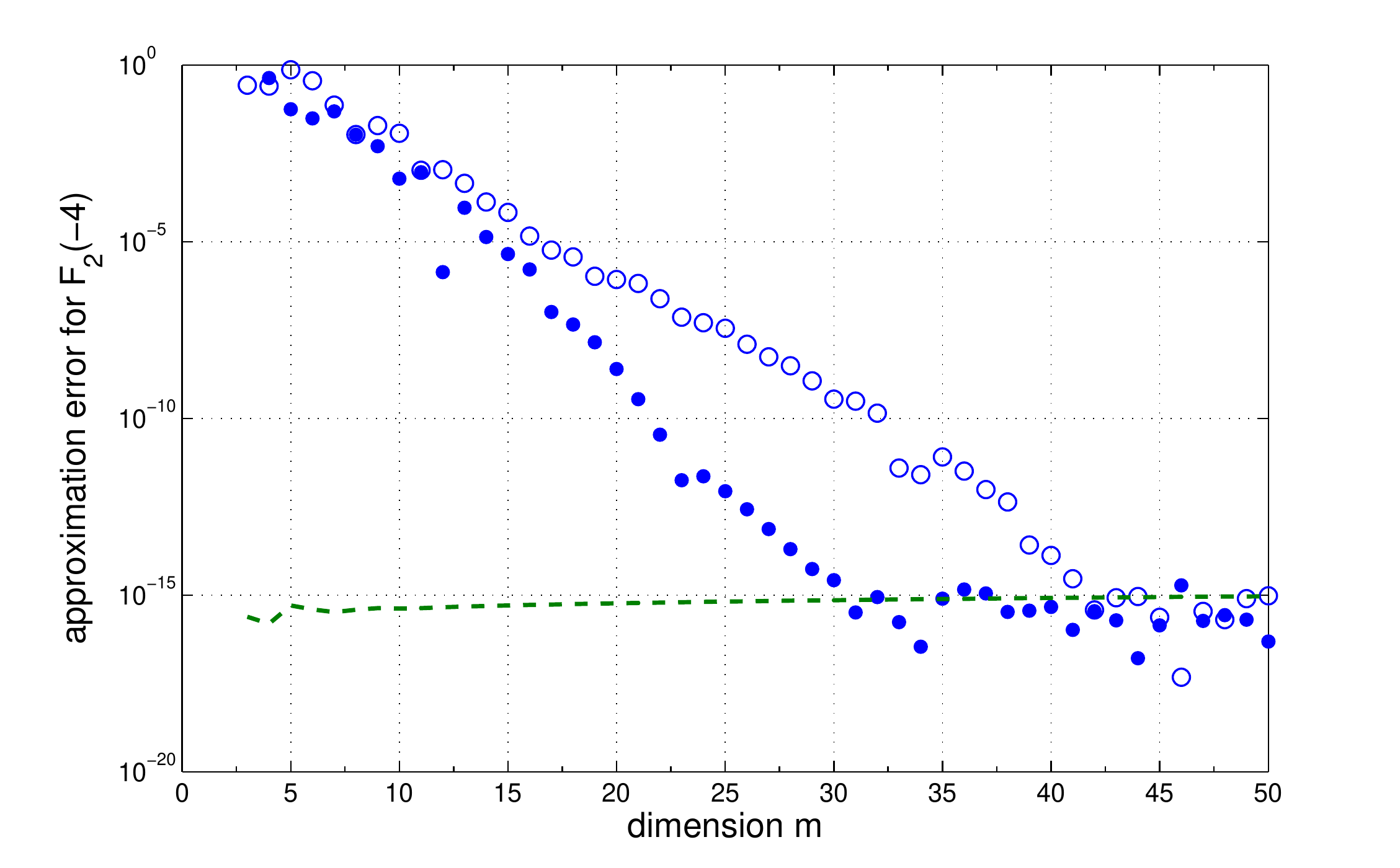}}
\end{minipage}
\end{center}\vspace*{-0.0625cm}
\caption{Convergence of the $m$-dimensional Nyström-type approximation---using the transformation (\ref{eq:transform})---of the Fredholm determinants $F_2(-2)$ (left) and $F_2(-4)$ (right),
based on Gauss--Legendre (dots) and Curtis--Clenshaw (circles).
The dashed line shows the amount, according to (\ref{eq:roundoff2}), of roundoff error due to the numerical evaluation of the finite-dimensional determinants; all calculations were done in IEEE double arithmetic
($\epsilon = 2^{-53}$).}
\label{fig:F2err}
\end{figure}

\noindent
From the general-purpose point of view, we recommend the third option. It is straightforward and does not require any specific knowledge, or construction, as would
the first and second option.

\smallskip
\paragraph{Remarks on other numerical methods to evaluate $F_2(s)$} As for the sine kernel, there is a selfadjoint second-order ordinary differential operator commuting with $A_s$ \cite[p.~166]{MR1257246}.
Though
this has been used to derive some asymptotic formulas, nothing is known in terms of special functions that would enable us to base a Ritz--Galerkin method on it. As \citeasnoun[p.~453]{MR2129906} puts it: ``In the case of the Airy kernel \ldots\ the differential equation
did not receive much attention and its solutions are not known.''

Prior to our work of calculating $F_2(s)$  directly from its determinantal expression, all the published numerical calculations started with \possessivecite{MR1257246} remarkable discovery of expressing $F_2(s)$ in terms
of the second Painlevé transcendent; namely
\[
F_2(s) = \exp\left(-\int_s^\infty (z-s) q(z)^2 \,dz\right)
\]
with $q(z)$ being the Hastings--McLeod \citeyear{MR555581} solution of Painlevé II,
\begin{equation}\label{eq:hastings}
q''(z) = 2 q(z)^3 + z\,q(z),\qquad q(z)\sim\Ai(z)\text{ \;as\; $z \to \infty$}.
\end{equation}
Initial value methods for the numerical integration of (\ref{eq:hastings}) suffer from severe stability problems \cite{MR2070096}. Instead, the numerically stable way of solving~(\ref{eq:hastings}) goes by considering $q(z)$ as a
connecting orbit, the other asymptotic state being
\[
q(z) \sim \sqrt\frac{-z}{2} \text{ \;as\; $z \to -\infty$},
\]
and using numerical two-point boundary value solvers \cite{Dieng05}.

\section{Extension to Systems of Integral Operators}\label{sect:matrixkernels}

We now consider an $N\times N$ system of integrals operators that is induced by continuous kernels $K_{ij} \in C(I_i\times I_j)$ ($i,j=1,\ldots,N$), where the $I_i \subset \R$
denote some finite intervals. The corresponding system of integral equations
\begin{equation}\label{eq:system}
u_i(x) + z\sum_{j=1}^N \int_{I_j} K_{ij}(x,y) u_j(y)\,dy = f_i(x)\qquad (x \in I_i,\; i,j=1,\ldots,N)
\end{equation}
defines, with $u=(u_1,\ldots,u_N)$ and $f=(f_1,\ldots,f_N)$, an operator equation
\[
u + z A u = f
\]
on the Hilbert space $\mathcal{H} = L^2(I_1) \oplus \cdots \oplus L^2(I_N)$.

\subsection{The Fredholm determinant for systems}\label{sect:sysdet}
Assuming $A$ to be  trace class, let us express $\det(I+z A)$ in terms of the system $(K_{ij})$ of kernels. To this end we show that the system (\ref{eq:system})
is equivalent to a {\em single} integral equation; an idea that, essentially,  can already be found in the early work of \citeasnoun[p.~388]{34.0422.02}.
To simplify notation, we assume that the $I_k$ are  \emph{disjoint} (a simple transformation of the system of integral equations by a set of translations will arrange for this).
We then have\footnote{The general case could be dealt with by  the topological sum, or coproduct, of the intervals $I_k$,
\[
\coprod_{k=1}^N I_k = \bigcup_{k=1}^N I_k \times \{k\}.
\]
One would then use \cite{MR2018275} the natural isometric isomorphism
\[
\mathcal{H} = \bigoplus_{k=1}^N L^2(I_k) \cong L^2\left(\coprod_{k=1}^N I_k \right).
\]}
\[
\mathcal{H} = \bigoplus_{k=1}^N L^2(I_k) \cong L^2(I), \qquad I = I_1 \cup \ldots \cup I_n.
\]
by means of the natural isometric isomorphism
\[
(u_1,\ldots,u_N) \mapsto u = \sum_{k=1}^N \chi_k u_k
\]
where $\chi_k$ denotes the characteristic function of the interval $I_k$. Given this picture, the operator $A$ can be viewed being the integral operator on $L^2(I)$
that is induced by the kernel
\[
K(x,y) = \sum_{i,j=1}^N \chi_i(x) K_{ij}(x,y) \chi_j(y).
\]
By (\ref{eq:detfred})
we finally get (cf. \citeasnoun[Thm~6.1]{MR1744872})
\begin{align*}
\det(I + z A) &= \sum_{n=0}^\infty \frac{z^n}{n!} \int_{I^n} \det\left(K(t_p,t_q)\right)_{p,q=1}^n \,dt_1\cdots\,dt_n \\*[2mm]
&= \sum_{n=0}^\infty \frac{z^n}{n!} \int_{I^n} \underbrace{\left(\sum_{i_1,\ldots,i_n=1}^N \chi_{i_1}(t_1) \cdots \chi_{i_n}(t_n)\right)}_{=1} \det\left(K(t_p,t_q)\right)_{p,q=1}^n \,dt_1\cdots\,dt_n \\*[2mm]
&= \sum_{n=0}^\infty \frac{z^n}{n!} \sum_{i_1,\ldots,i_n=1}^N \int_{I_{i_1}\times\cdots\times I_{i_n}} \det\left(K(t_p,t_q)\right)_{p,q=1}^n \,dt_1\cdots\,dt_n \\*[2mm]
&= \sum_{n=0}^\infty \frac{z^n}{n!} \sum_{i_1,\ldots,i_n=1}^N \int_{I_{i_1}\times\cdots\times I_{i_n}} \det\left(K_{i_p i_q}(t_p,t_q)\right)_{p,q=1}^n \,dt_1\cdots\,dt_n.
\end{align*}
By eventually transforming back to the originally given, non-disjoint intervals $I_k$, the last expression is the general formula that we have sought for: $\det(I+z A) = d(z)$ with
\begin{equation}\label{eq:systemdet}
d(z) = \sum_{n=0}^\infty \frac{z^n}{n!} \sum_{i_1,\ldots,i_n=1}^N \int_{I_{i_1}\times\cdots\times I_{i_n}} \det\left(K_{i_p i_q}(t_p,t_q)\right)_{p,q=1}^n \,dt_1\cdots\,dt_n.
\end{equation}
This is a perfectly well defined entire function for \emph{any} system $K_{ij}$ of continuous kernels, independently of whether $A$ is a trace class operator or not.
We call it the Fredholm determinant of the system.

\smallskip

\paragraph{The determinant of block matrices} In preparation of our discussion of Nyström-type methods for approximating (\ref{eq:systemdet}) we shortly discuss the determinant
of $N\times N$-block matrices
\[
A = \begin{pmatrix}
A_{11} & \cdots & A_{1N} \\*[1mm]
 \vdots & & \vdots \\*[1mm]
 A_{N1} & \cdots & A_{NN}
\end{pmatrix} \in \C^{M\times M},\qquad A_{ij} \in \C^{m_i \times m_j},\quad M = m_1 + \cdots + m_N.
\]
Starting with von Koch's formula (\ref{eq:vonKoch}), an argument\footnote{Alternatively, we can use (\ref{eq:detgrothen}) and, recursively, the ``binomial'' formula \cite[p.~121]{MR0224623}
\[
\bigwedge\nolimits^k(V_0 \oplus V_1) = \bigoplus_{j=0}^k \left(\bigwedge\nolimits^jV_0\right) \otimes \left(\bigwedge\nolimits^{k-j}V_1 \right)
\]
of exterior algebra, which is valid for general vector spaces $V_0$ and $V_1$.}
that is similar to the one that has led us to (\ref{eq:systemdet}) yields
\begin{equation}\label{eq:blockmatrixdet}
\det(I+z A) = \sum_{n=0}^\infty  \frac{z^n}{n!} \sum_{i_1,\ldots,i_n=1}^N \sum_{k_1=1}^{m_{i_1}} \cdots \sum_{k_n=1}^{m_{i_n}} \det\left((A_{i_p,i_q})_{k_p,k_q}\right)_{p,q=1}^n.
\end{equation}

\subsection{Quadrature methods for systems}

Given a quadrature formula for each of the intervals~$I_i$, namely
\begin{equation}\label{eq:quadsys}
Q_i(f) = \sum_{j=1}^{m_i} w_{ij} f(x_{ij}) \;\approx\; \int_{I_i} f(x)\,dx,
\end{equation}
we aim at generalizing the Nyström-type method of Section~\ref{sect:quad}. We restrict ourselves to the case of positive weights, $w_{ij}>0$, and generalize the method from the single operator case as given in (\ref{eq:detnysym}) to the
system case in the following form:
\begin{equation}\label{eq:nysysdef}
d_Q(z) = \det(I+z A_Q), \qquad A_Q = \begin{pmatrix}
A_{11} & \cdots & A_{1N} \\*[1mm]
 \vdots & & \vdots \\*[1mm]
 A_{N1} & \cdots & A_{NN}
\end{pmatrix}
\end{equation}
with the sub-matrices $A_{ij}$ defined by the entries
\[
(A_{ij})_{p,q} = w_{ip}^{1/2} K_{ij}(x_{ip},x_{jq}) w_{jq}^{1/2} \qquad (p=1,\ldots,m_i,\; q=1,\ldots,m_j).
\]
This can be as straightforwardly implemented as in the case of a single operator. Now, a convergence theory can be built on a representation of the error $d_Q(z)-d(z)$ that is analogous to (\ref{eq:nyerr}).
To this end we simplify the notation by introducing the following functions on $I_{i_1}\times\cdots\times I_{i_n}$,
\[
K_{i_1,\ldots,i_n}(t_1,\ldots,t_n) = \det\left(K_{i_p i_q}(t_p,t_q)\right)_{p,q=1}^n\,,
\]
and by defining, for functions $f$ on $I_{i_1}\times\cdots\times I_{i_n}$, the product quadrature formula
\begin{multline*}
\left(\prod_{k=1}^n Q_{i_k} \right)(f) = \sum_{j_1=1}^{m_{i_1}} \cdots   \sum_{j_n=1}^{m_{i_n}} w_{i_1j_1} \cdots w_{i_nj_n} f(x_{i_1j_1},\ldots,x_{i_nj_n}) \\*[2mm]
\approx \int_{I_{i_1}\times\cdots\times I_{i_n}} f(t_1,\ldots,t_n)\,dt_1\cdots\,dt_n.
\end{multline*}
Thus, we can rewrite the Fredholm determinant (\ref{eq:systemdet}) in the form
\[
d(z) = 1 + \sum_{n=1}^\infty \frac{z^n}{n!} \sum_{i_1,\ldots,i_n=1}^N \int_{I_{i_1}\times\cdots\times I_{i_n}}K_{i_1,\ldots,i_n}(t_1,\ldots,t_n) \,dt_1\cdots\,dt_n.
\]
Likewise, by observing the generalized von Koch formula (\ref{eq:blockmatrixdet}), we put the definition~(\ref{eq:nysysdef}) of $d_Q(z)$ to the form
\[
d_Q(z) = 1+  \sum_{n=1}^\infty \frac{z^n}{n!} \sum_{i_1,\ldots,i_n=1}^N \left(\prod_{k=1}^n Q_{i_k} \right)(K_{i_1,\ldots,i_n}).
\]
Thus, once again, the Nyström--type method amounts for approximating each multidimensional integral of the power series of the Fredholm determinant by using a product
quadrature rule. Given this representation, Theorem~\ref{thm:nyerr} can straightforwardly be generalized to the system case:

\smallskip

\begin{theorem} If $K_{ij} \in C^{k-1,1}(I_i \times I_j)$, then for each set (\ref{eq:quadsys}) of quadrature formulae of a common order $\nu \geq k$ with positive weights there holds the error estimate
\[
d_Q(z) - d(z) = O(\nu^{-k}) \qquad (\nu\to\infty),
\]
uniformly for bounded $z$.

If the $K_{ij}$ are bounded analytic on $\mathcal{E}_\rho(I_i) \times \mathcal{E}_\rho(I_j)$ (with the ellipse $\mathcal{E}_\rho(I_i)$ defined, with respect to $I_i$, as in Theorem~\ref{thm:quaderr}),
then for each set (\ref{eq:quadsys}) of quadrature formulae of a common order $\nu$ with positive weights there holds the error estimate
\[
d_Q(z)-d(z) = O(\rho^{-\nu})  \qquad (\nu\to\infty),
\]
uniformly for bounded $z$.
\end{theorem}

\subsection{Examples from random matrix theory}

Here, we apply the Nyström-type method (\ref{eq:nysysdef}) to two $2\times 2$-systems of integral operators that have recently been studied in random matrix theory.

\begin{figure}[tbp]
\begin{center}
\begin{minipage}{0.65\textwidth}
{\includegraphics[width=\textwidth]{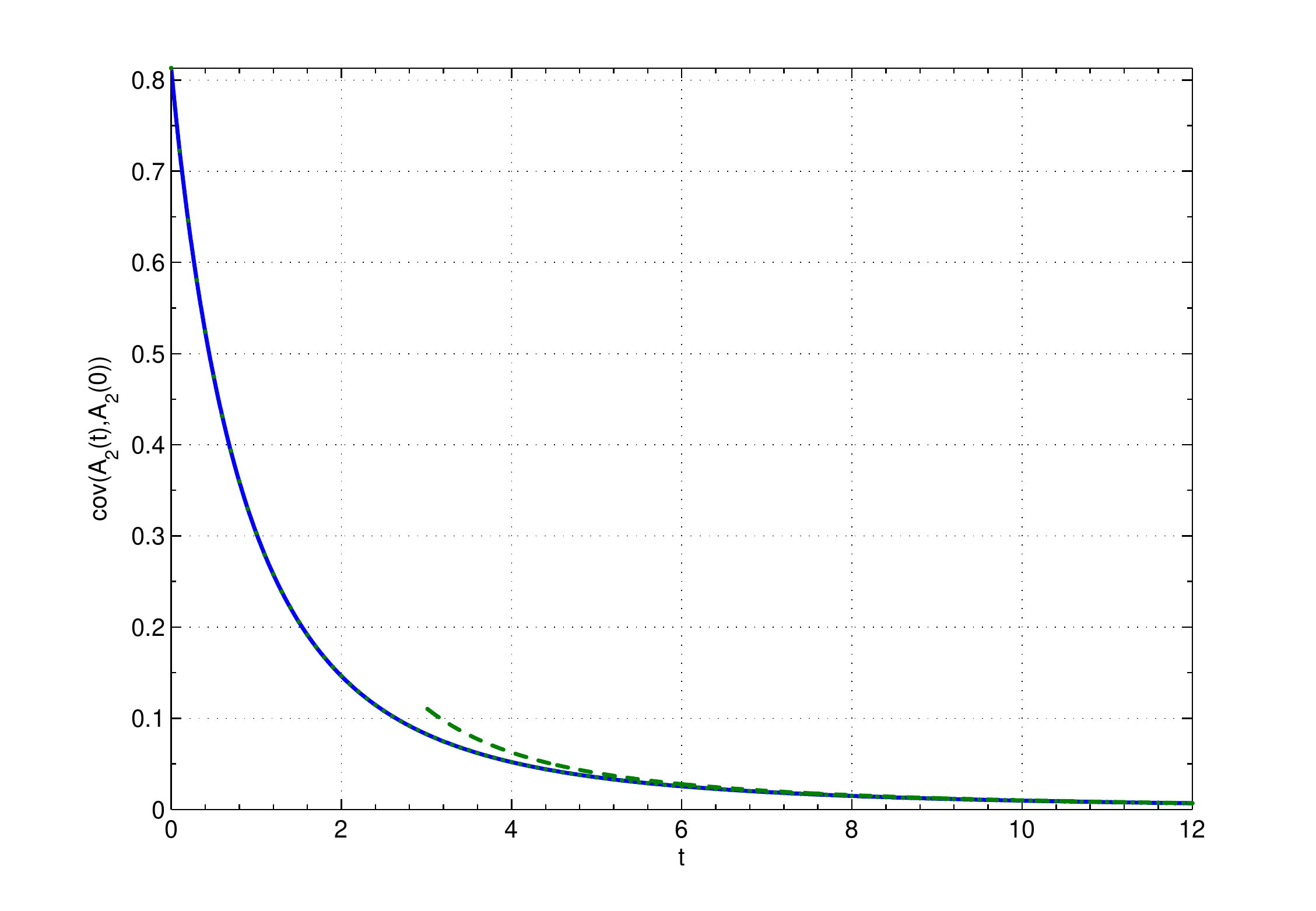}}
\end{minipage}
\end{center}\vspace*{-0.0625cm}
\caption{Values of the two-point correlation function $\cov(\mathcal{A}_2(t),\mathcal{A}_2(0))$ of the Airy process $\mathcal{A}_2(t)$ (solid line). The dashed line shows the
first term of the asymptotic expansion $\cov(\mathcal{A}_2(t),\mathcal{A}_2(0)) \sim t^{-2}$ as $t \to \infty$.}
\label{fig:Airy2}
\end{figure}

\smallskip
\paragraph{Two-point correlation of the Airy process}

The Airy process $\mathcal{A}_2(t)$ describes, in a properly rescaled limit of infinite dimension,
the maximum eigenvalue of Hermitian matrix ensemble whose entries develop according to the Ornstein--Uhlenbeck process. This stationary stochastic process was introduced by \citeasnoun{MR1933446} and further studied
by \citeasnoun{MR2018275}. These authors have shown that the joint probability function is given by a Fredholm determinant; namely
\begin{equation}\label{eq:jointprobairy2}
\mathbb{P}(\mathcal{A}_2(t) \leq s_1, \mathcal{A}_2(0) \leq s_2) = \det\left(I -
\begin{pmatrix}
A_0 & A_t \\*[1mm]
A_{-t} & A_0
\end{pmatrix}{\projected{L^2(s_1,\infty)\oplus L^2(s_2,\infty)}}\right)
\end{equation}
with integral operators $A_t$ that are induced by the kernel functions
\begin{equation}\label{eq:airy2kernel}
K_t(x,y) = \begin{cases}
\phantom{-}\displaystyle\int_0^\infty e^{-\xi t} \Ai(x+\xi)\Ai(y+\xi)\,d\xi, &\qquad t> 0, \\*[4mm]
-\displaystyle\int_{-\infty}^0 e^{-\xi t} \Ai(x+\xi)\Ai(y+\xi)\,d\xi, &\qquad \text{otherwise}.
\end{cases}
\end{equation}
Of particular interest is the two-point correlation function
\begin{align}\label{eq:twopintairy2}
\cov(\mathcal{A}_2(t),\mathcal{A}_2(0)) &= \mathbb{E}(\mathcal{A}_2(t)\mathcal{A}_2(0))- \mathbb{E}(\mathcal{A}_2(t)) \mathbb{E}(\mathcal{A}_2(0))\\*[2mm]
&= \int_{\R^2} s_1 s_2 \frac{\partial^2\mathbb{P}(\mathcal{A}_2(t) \leq s_1, \mathcal{A}_2(0) \leq s_2)}{\partial s_1 \partial s_2} \,ds_1ds_2 -c_1^2,\notag
\end{align}
where $c_1$ denotes the expectation value of the Tracy--Widom distribution (\ref{eq:F2}).
We have calculated
 this correlation function for $0 \leq t \leq 100$ in steps of $0.1$ to an absolute error of $\pm 10^{-10}$, see Figure~\ref{fig:Airy2}.\footnote{A table can be obtained from the author upon request.
 \citeasnoun[Fig.~2]{Sasa05}
 shows a plot (which differs by a scaling factor of two in both the function value and the time $t$) of the
  closely related function
 \[
g_2(t) = \sqrt{\var(\mathcal{A}_2(t)-\mathcal{A}_2(0))/2} = \sqrt{\var(\mathcal{A}_2(0)) - \cov(\mathcal{A}_2(t),\mathcal{A}_2(0))}
 \] ---without, however, commenting on either the numerical procedure used
 or on the accuracy obtained.}
Here are some details about the numerical procedure:

\smallskip
\begin{itemize}
\item Infinite intervals of integration, such as in the definition (\ref{eq:airy2kernel}) of the kernels  or for the domain of the integral operators (\ref{eq:jointprobairy2}) themselves,
are handled by a transformation to the finite interval $[0,1]$ as in Section~\ref{sect:airy}.\\*[-3mm]
\item The kernels (\ref{eq:airy2kernel}) are evaluated, after transformation, by a Gauss--Legendre quadrature.\\*[-3mm]
\item The joint probability distribution (\ref{eq:jointprobairy2}) is then evaluated, after transformation, by the Nyström-type method of this section, based on Gauss--Legendre quadrature.\\*[-3mm]
\item To avoid numerical differentiation, the expectation values defining the two-point correlation (\ref{eq:twopintairy2}) are evaluated by truncation of the integrals, partial integration, and using
a Gauss--Legendre quadrature once more.
\end{itemize}

\smallskip
\noindent
Because of analyticity, the convergence is always exponential.
With parameters carefully (i.e., adaptively) adjusted to deliver an absolute error of $\pm 10^{-10}$, the evaluation of the two-point correlation takes, for a single time $t$
and using a 2~GHz PC, about 20 minutes on average. The results were
cross-checked, for small $t$, with the asymptotic expansion \cite{MR1933446,Hagg07}
\begin{multline*}
\cov(\mathcal{A}_2(t),\mathcal{A}_2(0)) = \var(\mathcal{A}_2(0)) - \tfrac{1}{2}\var(\mathcal{A}_2(t)-\mathcal{A}_2(0))\\*[1mm]
 = \var(\mathcal{A}_2(0))\,  -\,t \,+ \,  O(t^2)\qquad (t\to0),
\end{multline*}
and, for large $t$, with the asymptotic expansion\footnote{\citeasnoun{MR2150191} have derived this asymptotic expansion from the masterfully obtained result that
$G(t,x,y) = \log\mathbb{P}(\mathcal{A}_2(t)\leq x, \mathcal{A}_2(0) \leq y)$ satisfies the following nonlinear 3rd order PDE with certain (asymptotic) boundary conditions:
\begin{multline*}
t \frac{\partial}{\partial t}\left(\frac{\partial^2}{\partial x^2}- \frac{\partial^2}{\partial y^2}\right) G =
\frac{\partial^3 G}{\partial x^2\partial y}\left(2\frac{\partial^2 G}{\partial y^2}+ \frac{\partial^2 G}{\partial x\partial y}-\frac{\partial^2 G}{\partial x^2}+x-y-t^2\right)\\*[2mm]
-\frac{\partial^3 G}{\partial y^2\partial x}\left(2\frac{\partial^2 G}{\partial x^2}+ \frac{\partial^2 G}{\partial x\partial y}-\frac{\partial^2 G}{\partial y^2}-x+y-t^2\right)
+\left(\frac{\partial^3 G}{\partial x^3}\frac{\partial}{\partial y}-\frac{\partial^3 G}{\partial y^3}\frac{\partial}{\partial x}\right)
\left(\frac{\partial}{\partial x}+\frac{\partial}{\partial y}\right) G.
\end{multline*}
The reader should contemplate a numerical calculation of the two-point correlation based on this PDE, rather than directly treating the Fredholm determinant as suggested by us.}
 \cite{MR2054175,MR2150191}
\[
\cov(\mathcal{A}_2(t),\mathcal{A}_2(0)) = t^{-2} +  c t^{-4} + O(t^{-6})\qquad (t\to\infty),
\]
where the constant $c=-3.542\cdots$ can explicitly be expressed in terms of the Hastings--McLeod solution (\ref{eq:hastings}) of Painlevé II.

\begin{figure}[tbp]
\begin{center}
\begin{minipage}{0.65\textwidth}
{\includegraphics[width=\textwidth]{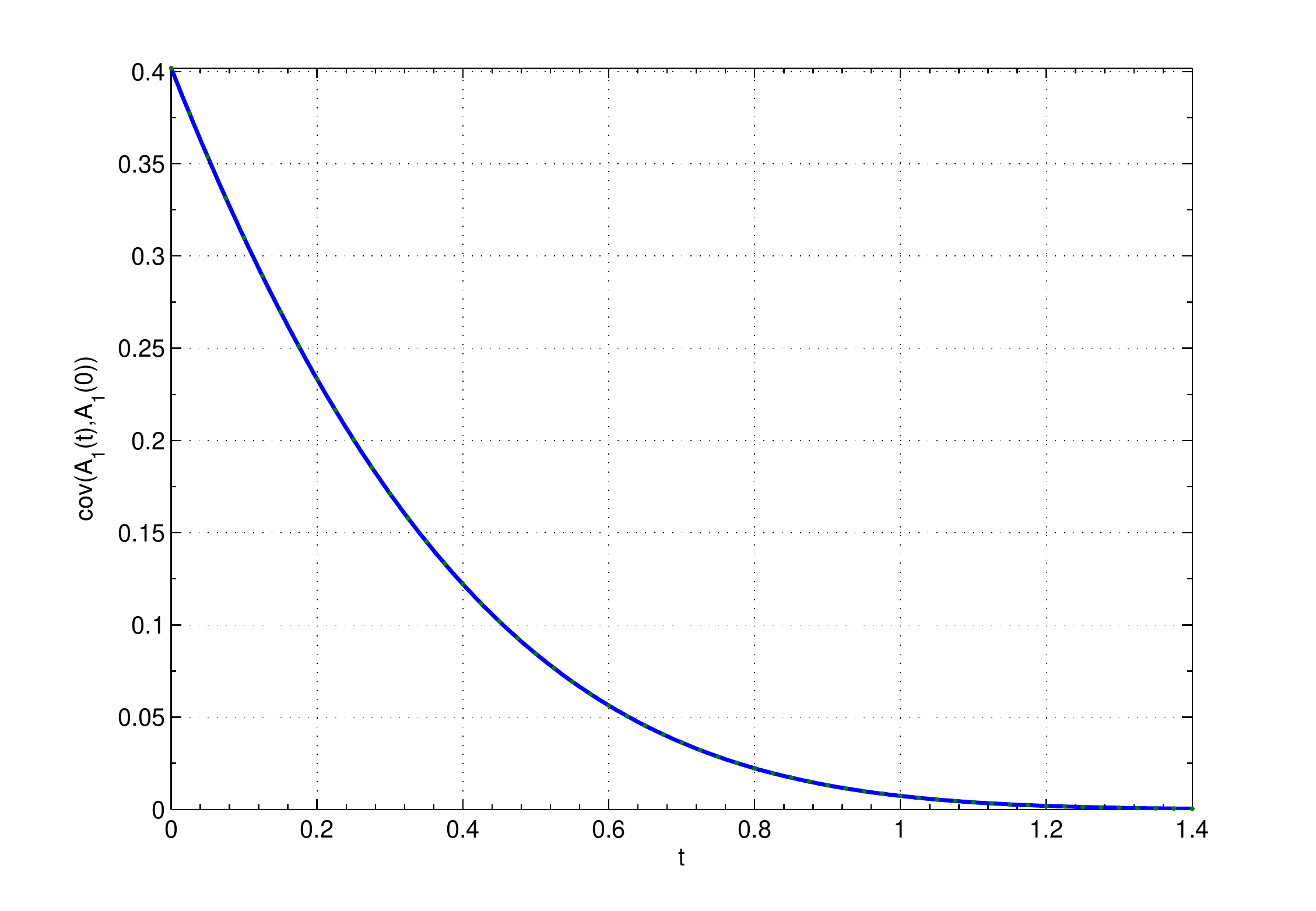}}
\end{minipage}
\end{center}\vspace*{-0.0625cm}
\caption{Values of the two-point correlation function $\cov(\mathcal{A}_1(t),\mathcal{A}_1(0))$ of the $\text{Airy}_1$ process $\mathcal{A}_1(t)$.}
\label{fig:Airy1}
\end{figure}

\medskip
\paragraph{Two-point correlation of the $\text{Airy}_1$ process} \citeasnoun{Sasa05} and \citeasnoun{MR2363389} have introduced the $\text{Airy}_1$ process $\mathcal{A}_1(t)$  for which,
once again,
the joint probability distribution can be given in terms of a Fredholm determinant; namely
\[
\mathbb{P}(\mathcal{A}_1(t) \leq s_1, \mathcal{A}_1(0) \leq s_2) = \det\left(I -
\begin{pmatrix}
A_0 & A_t \\*[1mm]
A_{-t} & A_0
\end{pmatrix}{\projected{L^2(s_1,\infty)\oplus L^2(s_2,\infty)}}\right)
\]
with integral operators $A_t$ that are now induced by the kernel functions
\[
K_t(x,y) = \begin{cases}
\Ai(x+y+t^2) e^{t(x+y)+2t^3/3} - \displaystyle\frac{\exp(-(x-y)^2/(4t))}{\sqrt{4\pi t}}, & \;t> 0, \\*[4mm]
\Ai(x+y+t^2) e^{t(x+y)+2t^3/3}, &\;\text{otherwise}.
\end{cases}
\]
By basically employing  the same numerical procedure as for the Airy process, we have succeeded in calculating
 the two-point correlation function $\cov(\mathcal{A}_1(t),\mathcal{A}_1(0))$ for $0 \leq t \leq 2.5$ in steps of $0.025$ to an absolute error of $\pm 10^{-10}$, see
 Figure~\ref{fig:Airy1}.
\footnote{A table can be obtained from the author upon request.
 \citeasnoun[Fig.~2]{Sasa05}
 shows a plot (which differs by a scaling factor of two in both the function value and the time $t$) of the
 closely related function
 \[
g_1(t) = \sqrt{\var(\mathcal{A}_1(t)-\mathcal{A}_1(0))/2} = \sqrt{\var(\mathcal{A}_1(0)) - \cov(\mathcal{A}_1(t),\mathcal{A}_1(0))}
 \] ---without, however, commenting on either the numerical procedure used
 or on the accuracy obtained.}
 For a single time $t$ the evaluation takes about 5 minutes on average (using a 2~GHz PC). This numerical result has been used by \citeasnoun{0806:3410} as a strong evidence
that the $\text{Airy}_1$ process is, unlike previously conjectured, \emph{not} the limit of the largest eigenvalue in GOE matrix diffusion.

\renewcommand{\thesection}{A}
\section{Appendices}
\subsection{Quadrature Rules}\label{app:quad}

For the ease of reference, we collect in this appendix some classical facts about quadrature rules in one and more dimensions.

\smallskip

\paragraph{Quadrature rules in one dimension}
We consider quadrature rules of the form
\begin{equation}\label{eq:quad1}
Q(f) = \sum_{j=1}^m w_j f(x_j)
\end{equation}
which are meant to approximate $\int_a^b f(x)\,dx$ for continuous functions $f$ on some finite interval $[a,b]\subset \R$.
We define the norm of a quadrature rule by
\[
\|Q\| = \sum_{j=1}^m |w_j|
\]
Convergence of a sequence of quadrature rules is characterized by the following theorem of Pólya \cite[p.~130]{MR760629}.

\smallskip

\begin{theorem}\label{thm:polya}
A sequence $Q_n$ of quadrature rules converges for continuous functions,
\[
\lim_{n\to\infty} Q_n(f) = \int_a^b f(x)\,dx \qquad (f \in C[a,b]),
\]
if and only if the sequence $\|Q_n\|$ of norms is bounded by some stability constant $\Lambda$ and if
\begin{equation}\label{eq:monomial}
\lim_{n\to\infty} Q_n(x^k) = \int_a^b x^k\,dx \qquad (k =0,1,2,\ldots).
\end{equation}
If the weights are all positive, then (\ref{eq:monomial}) already implies the boundedness of  $\|Q_n\|=Q_n(1)$.
\end{theorem}

\smallskip

A quadrature rule $Q$ is of order $\nu\geq 1$, if it is exact for all polynomials of degree at most $\nu-1$. Using results from the theory polynomial best approximation
one can prove
quite strong error estimates \cite[§4.8]{MR760629}.

\smallskip

\begin{theorem}\label{thm:quaderr}
If $f \in C^{k-1,1}[a,b]$, then for each quadrature rule $Q$ of order $\nu\geq k$ with positive weights
there holds the error estimate
\[
\left| Q(f) - \int_a^b f(x)\,dx \right| \leq c_k\, (b-a)^{k+1} \nu^{-k} \|f^{(k)}\|_{\scriptscriptstyle L^\infty(a,b)}\,,
\]
with a constant\footnote{Taking Jackson's inequality as given in \citeasnoun[p.~147]{MR1656150}, $c_k = 2 (\pi e/4)^k/\sqrt{2\pi k}$ will do the job.} $c_k$ depending only on $k$.

If $f$ is bounded analytic in the ellipse $\mathcal{E}_\rho$ with foci at $a$, $b$ and semiaxes of lengths $s > \sigma$ such that
\[
\rho = \sqrt\frac{s+\sigma}{s-\sigma}\,,
\]
then for each quadrature rule $Q$ of order $\nu$ with positive weights there holds the error estimate
\[
\left| Q(f) - \int_a^b f(x)\,dx \right| \leq \frac{4(b-a)\rho^{-\nu} }{1-\rho^{-1}}  \|f\|_{\scriptscriptstyle L^\infty(\mathcal{E}_\rho)}.
\]

\end{theorem}

\smallskip

\paragraph{Quadrature rules in two and more dimensions} For the $n$-dimensional integral
\[
\int_{[a,b]^n} f(t_1,\ldots,t_n) \,dt_1 \cdots\,dt_n
\]
we consider the product quadrature rule $Q^n$ that is induced by an one dimensional quadrature rule $Q$ of the form (\ref{eq:quad1}), namely
\begin{equation}\label{eq:productrule}
Q^n(f) = \sum_{j_1,\ldots,j_n=1}^m w_{j_1}\cdots w_{j_n} \,f(x_{j_1},\ldots,x_{j_n}).
\end{equation}
We introduce some further notation for two classes of functions $f$. First, for $f \in C^{k-1,1}([a,b]^n)$, we define the seminorm
\begin{equation}\label{eq:seminormk}
|f|_k = \sum_{i=1}^n \|\partial_i^k f\|_{\scriptscriptstyle L^\infty((a,b)^n)}.
\end{equation}
Second, if $f \in C([a,b]^n)$ is sectional analytic---that is, analytic with respect to each variable $t_i$ while the other variables are fixed in $[a,b]$---in the ellipse $\mathcal{E}_\rho$
(defined in Theorem~\ref{thm:quaderr}), and if $f$ is uniformly bounded there, we call $f$ to be of class $\mathcal{C}_\rho$ with norm
\begin{equation}\label{eq:normcrho}
\|f\|_{\scriptscriptstyle \mathcal{C}_\rho} = \sum_{i=1}^n \;\max_{(t_1,\ldots,t_{i-1},t_{i+1},\ldots,t_n)\in[a,b]^{n-1}}
\|f(t_1,\ldots,t_{i-1},\cdot\,,t_{i+1},\ldots,t_n)\|_{\scriptscriptstyle L^\infty(\mathcal{E}_\rho)}.
\end{equation}
By a straightforward reduction argument \cite[p.~361]{MR760629} to the quadrature errors of the one-dimensional coordinate sections of $f$, Theorems~\ref{thm:polya} and~\ref{thm:quaderr} can
now be generalized to $n$ dimensions.

\smallskip
\begin{theorem}\label{thm:quaderrn} If a sequence of quadrature rules converges for continuous functions, then the same holds for the induced $n$-dimensional product rules.

If $f \in C^{k-1,1}([a,b]^n)$, then for each one-dimensional quadrature rule $Q$ of order $\nu\geq k$ with positive weights
there holds the error estimate
\[
\left| Q^n(f) - \int_{[a,b]^n} f(t_1,\ldots,t_n) \,dt_1 \cdots\,dt_n \right| \leq c_k\, (b-a)^{n+k}\, \nu^{-k} |f|_k\,,
\]
with the same constant $c_k$ depending only on $k$ as in Theorem~\ref{thm:quaderr}.

If $f\in C([a,b]^n)$ is of class $\mathcal{C}_\rho$, then for each one-dimensional quadrature rule $Q$ of order $\nu$ with positive weights there holds the error estimate
\[
\left| Q^n(f) - \int_{[a,b]^n} f(t_1,\ldots,t_n) \,dt_1 \cdots\,dt_n \right| \leq \frac{4(b-a)^{n} \rho^{-\nu}}{1-\rho^{-1}}  \|f\|_{\scriptscriptstyle \mathcal{C}_\rho}.
\]
\end{theorem}

\smallskip

\paragraph{Notes on Gauss--Legendre and Curtis--Clenshaw quadrature}

Arguably, the most interesting families of quadrature rules, with \emph{positive} weights, are the Clenshaw--Curtis and Gauss-Legendre rules.
With $m$ points, the first is of order $\nu=m$, the second of order $\nu=2m$. Thus, Theorems~\ref{thm:polya} and~\ref{thm:quaderr} apply.
The cost of computing the weights and points of Clenshaw--Curtis is $O(m\log m)$ using FFT, that of Gauss--Legendre is $O(m^2)$ using the Golub--Welsh algorithm;
for details see \cite{MR2214855} and \cite{Tref08}. The latter paper studies in depth the reasons why the Clenshaw--Curtis rule, despite having
only half the order, performs essentially as well as Gauss--Legendre for most integrands. To facilitate reproducibility  we offer
the Matlab code (which is just a minor variation of the code given in the papers mentioned above) that has been used in our numerical experiments:

\medskip\begin{quote}
\begin{verbatim}
function [w,c] = ClenshawCurtis(a,b,m)
m = m-1;
c = cos((0:m)*pi/m);
M = [1:2:m-1]'; l = length(M); n = m-l;
v0 = [2./M./(M-2); 1/M(end); zeros(n,1)];
v2 = -v0(1:end-1)-v0(end:-1:2);
g0 = -ones(m,1); g0(1+l)=g0(1+l)+m; g0(1+n)=g0(1+n)+m;
g = g0/(m^2+mod(m,2));
w = ifft(v2+g); w(m+1) = w(1);
c = ((1-c)/2*a+(1+c)/2*b)';
w = ((b-a)*w/2)';
\end{verbatim}
\end{quote}\medskip

\noindent for Clenshaw--Curtis; and

\medskip\begin{quote}
\begin{verbatim}
function [w,c] = GaussLegendre(a,b,m)
k = 1:m-1; beta = k./sqrt((2*k-1).*(2*k+1));
T = diag(beta,-1) + diag(beta,1);
[V,L] = eig(T);
c = (diag(L)+1)/2; c = (1-c)*a+c*b;
w = (b-a)*V(1,:).^2;
\end{verbatim}
\end{quote}\medskip

\noindent  for Gauss--Legendre, respectively.
Note, however, that the code for Gauss--Legendre is, unfortunately, suboptimal in requiring $O(m^3)$ rather than $O(m^2)$ operations, since it establishes the full matrix $V$ of eigenvectors of
the Jacobi matrix $T$ instead of directly calculating just their first components $V(1,:)$  as in the fully fledged Golub--Welsh algorithm. Even then, there may well be more
accurate, and more efficient, alternatives of computing the points and weights of Gauss--Legendre quadrature, see the discussions in \citeasnoun[§2]{MR1808574} and \citeasnoun[§4]{MR1950519} and the literature cited therein.

\subsection{Determinantal bounds}

In Section~\ref{sect:quad}, for a continuous kernel $K \in C([a,b]^2)$ of an integral operator, we need some bounds on the derivatives of the induced $n$-dimensional function
\[
K_n(t_1,\ldots t_n) = \det\left(K(t_p,t_q)\right)_{p,q=1}^n.
\]
To this end, if $K \in C^{k-1,1}([a,b]^2)$ we define the norm
\begin{equation}\label{eq:Knormk}
\|K\|_k = \max_{i+j\leq k} \|\partial_1^i\partial_2^j K\|_{\scriptscriptstyle L^\infty}.
\end{equation}

\begin{lemma}\label{lem:Kn} If $K \in C([a,b]^2)$, then $K_n \in C([a,b]^n)$ with
\begin{equation}\label{eq:Knbound}
\|K_n\|_{\scriptscriptstyle L^\infty} \leq n^{n/2} \|K\|_{\scriptscriptstyle L^\infty}^n.
\end{equation}
If $K \in C^{k-1,1}([a,b]^2)$, then $K_n \in C^{k-1,1}([a,b]^n)$ with the seminorm (defined in (\ref{eq:seminormk}))
\begin{equation}\label{eq:Knboundk}
|K_n|_k \leq 2^k n^{(n+2)/2} \|K\|_k^n.
\end{equation}
If $K$ is bounded analytic on $\mathcal{E}_\rho \times \mathcal{E}_\rho$ (with the ellipse $\mathcal{E}_\rho$ defined in Theorem~\ref{thm:quaderr}), then $K_n$ is of class $\mathcal{C}_\rho$ (defined in (\ref{eq:normcrho}))
and satisfies
\begin{equation}\label{eq:KnboundCrho}
\|K_n\|_{\scriptscriptstyle \mathcal{C}_\rho} \leq n^{(n+2)/2} \|K\|_{\scriptscriptstyle L^\infty(\mathcal{E}_\rho \times \mathcal{E}_\rho)}^n.
\end{equation}
\end{lemma}

\begin{proof}
Using the multilinearity of the determinant we have
\begin{multline*}
\frac{\partial^k}{\partial t_i^k}  \frac{\partial^l}{\partial s_j^l}
\begin{vmatrix}
K(t_1,s_1) & \cdots & K(t_1,s_j) & \cdots & K(t_1,s_n) \\*[1mm]
\vdots & & \vdots & & \vdots \\*[1mm]
K(t_i,s_1) & \cdots & K(t_i,s_j) & \cdots & K(t_i,s_n) \\*[1mm]
\vdots & & \vdots & & \vdots \\*[1mm]
K(t_n,s_1) & \cdots & K(t_n,s_j) & \cdots & K(t_n,s_n)
\end{vmatrix} \\*[2mm]
= \begin{vmatrix}
K(t_1,s_1) & \cdots & \partial_2^l K(t_1,s_j) & \cdots & K(t_1,s_n) \\*[1mm]
\vdots & & \vdots & & \vdots \\*[1mm]
\partial_1^k K(t_i,s_1) & \cdots & \partial_1^k\partial_2^lK(t_i,s_j) & \cdots & \partial_1^kK(t_i,s_n) \\*[1mm]
\vdots & & \vdots & & \vdots \\*[1mm]
K(t_n,s_1) & \cdots & \partial_2^lK(t_n,s_j) & \cdots & K(t_n,s_n)
\end{vmatrix},
\end{multline*}
which is, by Hadamard's inequality \cite[p.~469]{MR17382},\footnote{This inequality, discovered by \citename{Hada93} in 1893, was already of fundamental importance to Fredholm's original theory \cite[p.~41]{Fred00}.} bounded by
the expression (see also \citeasnoun[p.~262]{MR1892228})
\[
n^{n/2}\left( \max_{i+j\leq k+l} \|\partial_1^i\partial_2^j K\|_{\scriptscriptstyle L^\infty}\right)^n.
\]
Now, with
\[
\partial_j^k K_n(t_1,\ldots,t_n) = \sum_{l=0}^k \binom{k}{l} \frac{\partial^{k-l}}{\partial t_j^{k-l}}  \frac{\partial^l}{\partial s_j^l}
\begin{vmatrix}
K(t_1,s_1) & \cdots & K(t_1,s_n) \\*[1mm]
\vdots & & \vdots \\*[1mm]
K(t_n,s_1) & \cdots & K(t_n,s_n)
\end{vmatrix}_{s_1=t_1,\ldots,s_n=t_n}
\]
we thus get
\[
\|\partial_j^k K_n\|_{\scriptscriptstyle L^\infty} \leq \sum_{l=0}^k \binom{k}{l} n^{n/2} \left( \max_{i+j\leq k} \|\partial_1^i\partial_2^j K\|_{\scriptscriptstyle L^\infty}\right)^n
= 2^k n^{n/2} \left( \max_{i+j\leq k} \|\partial_1^i\partial_2^j K\|_{\scriptscriptstyle L^\infty}\right)^n.
\]
This proves the asserted bounds (\ref{eq:Knbound}) and (\ref{eq:Knboundk}) with $k=0$ and $k\geq 1$, respectively. The class $\mathcal{C}_\rho$ bound  (\ref{eq:KnboundCrho}) follows analogously to the case $k=0$.
\end{proof}

\subsection{Properties of a certain function used in Theorem~\protect\ref{thm:nyerr}}

The power series
\begin{equation}\label{eq:phi}
\Phi(z) = \sum_{n=1}^\infty \frac{n^{(n+2)/2}}{n!}\, z^n
\end{equation}
defines an entire function on $\C$ (as the following lemma readily implies).

\smallskip
\begin{lemma}\label{lem:phi}
Let $\Psi$ be the entire function given by the expression
\[
\Psi(z) = 1 +\frac{\sqrt{\pi}}{2} z \,e^{z^2/4} \left(1+\erf\left(\frac z2\right)\right).
\]
If $x > 0$, then the series $\Phi(x)$ is enclosed by:\footnote{Note the sharpness of this enclosure: $\sqrt{e/\pi}=0.93019\cdots$.}
\[
\sqrt{\frac{e}{\pi}}\, x \,\Psi(x\sqrt{2e}) \leq \Phi(x) \leq x\, \Psi(x\sqrt{2e}).
\]
\end{lemma}
\begin{proof}
For $x>0$ we have
\[
\Phi(x) = x \sum_{n=1}^\infty \frac{n^{n/2}}{\Gamma(n)}x^{n-1}.
\]
By Stirling's formula and monotonicity we get for $n\geq 1$
\[
\sqrt{\frac{e}{\pi}} \leq \frac{n^{n/2}}{\Gamma((n+1)/2)\,(\sqrt{2e}\,)^{n-1}} \leq 1;
\]
in fact, the upper bound is obtained for $n=1$ and the lower bound for $n\to \infty$. Thus, by observing
\[
\sum_{n=1}^\infty \frac{\Gamma((n+1)/2)}{\Gamma(n)} z^{n-1} =  1 +\frac{\sqrt{\pi}}{2} z \,e^{z^2/4} \left(1+\erf\left(\frac z2\right)\right) = \Psi(z)
\]
we get the asserted enclosure.
\end{proof}

\section*{Acknowledgements}
It is a pleasure to acknowledge that this work
has started when I attended the programme on ``Highly Oscillatory Problems'' at the Isaac Newton Institute in Cambridge. It was a great opportunity meeting Percy Deift there, who
introduced me to numerical problems related to random matrix theory (even though he was envisioning a general
numerical treatment of Painlevé transcendents, but not of Fredholm determinants). I am grateful for his advice as
well as for the communication with Herbert Spohn and Michael Prähofer who directed my search for an ``open''
numerical problem to the two-point correlation functions of the Airy and $\text{Airy}_1$ processes. I thank Patrik Ferrari who pointed me to the \citeyear{Sasa05} paper
of \citename{Sasa05}.
Given the discovery that Theorem~\ref{thm:nyconv}, which was pivotal to my study, is essentially
a long forgotten (see my discussion on p.~\pageref{ex:intro}) result of Hilbert's 1904 work, I experienced much reconciliation---please allow me this very personal statement---from reading the poem ``East Coker'' (1940), in which
 T.~S.~Eliot, that ``radical traditionalist'', described the nature of the human struggle for progress in life, art, or science:

 \smallskip

 \begin{quote}
 \ldots\ And so each venture\\
 Is a new beginning \ldots\\
 \ldots\ And what there is to conquer\\
 By strength and submission, has already been discovered\\
 Once or twice, or several times, by men whom one cannot hope\\
 To emulate---but there is no competition---\\
 There is only the fight to recover what has been lost\\
 And found and lost again and again \ldots
 \end{quote}

%

\bibliographystyle{kluwer}
\bibliography{article}

\end{document}